\documentclass[11pt]{llncs}
\usepackage{epsfig,latexsym}
\usepackage{xcolor}
\usepackage{tikz} % tikz pictures 
\usepackage{graphicx}%stackengine
\usepackage{subfigure}

\usepackage{color}
\usepackage{algorithmic}
\usepackage{algorithm}
\usepackage{xfrac}
\usepackage{epsfig}
\usepackage{amssymb}
\usepackage{graphicx}
\usepackage{url}
\usepackage{amsmath}
\usepackage{hyperref}
\usepackage[square,comma,numbers]{natbib}%%%%%%%%%%%%%

\usepackage{color}
\newcommand{\commentout}[1]{}

\newcommand{\cT}{{\cal T}}

\newcommand{\tb}{{\sf tb}}
\newcommand{\tw}{{\sf tw}}
\newcommand{\tl}{{\sf tl}}
\newcommand{\itl}{{\sf itl}}
\newcommand{\itb}{{\sf itb}}

\newcommand{\BNC}{{\sf BNC}}
\newcommand{\bn}{{\sf bnc}}
\newcommand{\BDS}{{\sf BDS}}
\newcommand{\mc}{{\sf mcw}}
\newcommand{\adt}{{\sf adt}}
\newcommand{\ad}{{\sf ad}} 
\newcommand{\td}{{\sf td}} %\td
\newcommand{\cbc}{{\sf cbc}} %\cbc
\newcommand{\bgc}{{\sf bgc}}
\newcommand{\BGC}{{\sf BGC}} 
\newcommand{\stb}{{\sf stb}} %\stb(G)
\newcommand{\glc}{{\sf glc}} %\stb(G)
\newcommand{\br}{{\sf br}}  
\newcommand{\sh}{{\sf sh}}  
\newcommand{\ph}{{\sf ph}}  
\newcommand{\mf}{{\sf mf}}  %\mf
\newcommand{\p}{{\sf p}}  %\mf
\newcommand{\q}{{\sf q}}  %\mf

\begin{document}
\sloppy

%\title{Once more on graphs with bounded tree-length}
\title{Graph parameters that are coarsely equivalent to tree-length\thanks{\date*{\today}}
}

\author{Feodor F. Dragan}

\institute{Computer Science Department, Kent State University, Kent, Ohio,  USA \\
\email{dragan@cs.kent.edu}  
}
\maketitle
%\date{September 2024}

%\date*{\today}
\sloppy

\begin{abstract} Two graph parameters are said to be coarsely equivalent if they are within constant factors from each other for every graph $G$. 
Recently, several graph parameters were shown to be coarsely  equivalent to tree-length. Recall that the length of a tree-decomposition $\cT(G)$ of a graph $G$ is the largest diameter of a bag in  $\cT(G)$, and the  tree-length of $G$ is the minimum of the length, over all tree-decompositions of $G$. 
We present simpler and sometimes with better bounds proofs for those known in literature results and further extend this list of graph parameters coarsely equivalent to tree-length. 
Among other new results, we show that the tree-length of a graph $G$ is small if and only if for every bramble ${\cal F}$ (or every Helly family of connected subgraphs ${\cal F}$, or every Helly family of paths ${\cal F}$)  of $G$,  there is a disk in $G$ with  
small radius that intercepts all members of ${\cal F}$. Furthermore, the tree-length of a graph $G$ is small if and only if $G$ can be embedded with a 
small additive distortion to an unweighted tree with the same vertex set as in $G$ (not involving any Steiner points). Additionally, we introduce a new natural "bridging`` property for cycles,  
which generalizes a known property of cycles in chordal graphs, and show that it also coarsely defines the tree-length.   
 \medskip

 \medskip

\end{abstract}

\section{Introduction} 
We say that two graph parameters $\p$ and $\q$ are {\em coarsely equivalent} if 
there are two universal constants $\alpha>0$ and $\beta>0$ such that $\alpha\cdot \q(G)\le \p(G)\le \beta\cdot \q(G)$ for every graph $G$.  So, if one parameter is bounded by a constant, then the other is bounded by a constant, too. 
Coarse equivalency of two graph parameters is useful in at least two scenarios. First, if one parameter is easier to compute, then it provides an easily computable constant-factor approximate for possibly hard to compute other parameter. This is the case for parameters the {\em cluster-diameter of a layering partition} and the {\em tree-length} of a graph (that are subjects of this paper; formal definitions of these and other parameters can be found in Section \ref{sec:prelim-on-par}). It is known \cite{AbDr16,DDGY-spanners,Dorisb2007,tree-spanner-appr}  that those two parameters are within small constant factors from each other. %, i.e., they are coarsely equivalent. 
The cluster-diameter of a layering partition can easily be computed,  and a layering partition serves as a crucial tool in designing an efficient % linear-time 
3-approximation algorithm for computing the tree-length of a graph (and its tree-decomposition with minimum length), which is NP-hard to compute exactly (see Section \ref{sec:layer-partit} and Section \ref{sec:tl} for details). 
Secondly, since a constant bound on one parameter implies a constant bound on the other, one can choose out of two a most suitable (a right one) parameter when designing FPT (approximation) algorithms for some particular optimization problems on bounded parameter graphs. For example, an FPT algorithm for the so-called {\em metric dimension problem} on bounded tree-length graphs developed in \cite{BFGR2017}  or sparse spanner  results obtained for bounded tree-length graphs in \cite{DDGY-spanners} are  useful and hold also for graphs with bounded cluster-diameter of a layering partition.  Approximation algorithms (whose approximation error bounds depend on the cluster-diameter of a layering partition) for the connected $r$-domination problem and the connected $p$-center problem developed in \cite{par-conn-p-c} are useful also for graphs with bounded tree-length. 

Layering partition and its cluster-diameter were used also in obtaining a 6-approximation algorithm for the problem of optimal non-contractive embedding of an unweighted graph metric into a weighted tree metric. This was possible because the cluster-diameter of a layering partition and the minimum distortion of an embedding of a graph into a tree (the so-called {\em tree-distortion} parameter) are within small constant factors from each other (see \cite{AbDr16,ChepoiDNRV12,WG13-Dragan,tree-spanner-appr};   Section \ref{sec:embed-to-tree} provides some details). %, i.e., they are coarsely equivalent. 
The cluster-diameter of a layering partition and the tree distortion were the first two graph parameters shown to %proven to %that were shown to 
be coarsely equivalent with  the 
tree-length of a graph.

This paper is inspired by recent insightful papers  \cite{BerSey2024,GeorPapa2023} and \cite{Diestel++}. They added several additional  graph parameters to the list of parameters that are coarsely equivalent to tree-length. Among other results, \cite{BerSey2024} showed that the tree-length of a graph $G$ is bounded if and only if there is an {\em $(L,C)$-quasi-isometry} (equivalently, an {\em $(1,C')$-quasi-isometry}) to a (weighted and with Steiner points) tree with $L,C$ ($C'$, respectively) bounded (see Section \ref{sec:embed-to-tree} for more details). There were results already known that characterize when a graph is quasi-isometric to a tree \cite{ChepoiDNRV12,Kerr,manning}. For example, a theorem of Manning for geodesic metric spaces (see \cite{manning}) implies that a graph is quasi-isometric to a tree if and only if its {\em bottleneck constant} is bounded. Authors of \cite{BerSey2024} give a graph theoretic proof for this (see Section \ref{sec:early-bn} for more details).  
One of the main results of \cite{BerSey2024} is a proof of Rose McCarty's  conjecture that the tree-length of a graph is small if and only if its {\em McCarty-width} is small (see Section \ref{sec:early-mcw} for more details). In a quest to find a cycle property that coarsely describes the tree-length, \cite{BerSey2024} introduced a new notion of bounded load geodesic cycles (see Section \ref{sec:glc} for the definition and some details). It was shown \cite{BerSey2024}  that the tree-length of a graph $G$ is bounded if and only if  
every geodesic loaded cycle of $G$ has bounded load. Recently, in \cite{GeorPapa2023}, the Manning's Theorem was  extended to include also a characterization via $K$-fat $K_3$-minors (see Section \ref{sec:early-bn} for a definition). It was shown \cite{GeorPapa2023} that the bottleneck constant of a graph $G$ is bounded by a constant if and only if $G$ has no $K$-fat $K_3$-minor for some constant $K>0$. 

%In this paper, w 
We give an extended overview of those existing results in Section \ref{sec:prelim-on-par}. In Section \ref{sec:proofs}, 
we incorporate the  cluster-diameter of a layering partition into the repertoire (the first graph parameter known to be coarsely equivalent to tree-length) and use it to simplify some proofs and, in some cases, get even better coarseness bounds than in \cite{BerSey2024,GeorPapa2023} (see Section \ref{sec:bnc}, Section \ref{sec:mcw}, and Section \ref{sec:adt}, Section \ref{sec:fat}).  Our proofs are  simpler and more direct. One of our results shows that the tree-length of a graph $G$ is bounded if and only if there is an {\em $(1,C)$-quasi-isometry}, with $C$ bounded, to  
an unweighted tree with the same vertex set as in $G$  
(a most restrictive quasi-isometry to a tree; see Section \ref{sec:adt}). As a byproduct, we also obtain several alternative proofs of McCarty's conjecture (see Theorem \ref{th:mcw-delta-rho}, Corollary   \ref{cor:mcw-adt}, and Theorem \ref{th:mf-tl-mcw}). 
More importantly, we   
add a number of new graph parameters that are coarsely equivalent to tree-length. Any result obtained for graphs with bounded tree-length automatically applies/extends to graphs with bounded such parameters. % that are coarsely equivalent to tree-length. 
Among other results, we show that the tree-length of a graph $G$ is bounded if and only if for every bramble ${\cal F}$ (or every Helly family of connected subgraphs ${\cal F}$, or every Helly family of paths ${\cal F}$)  of $G$,  there is a disk in $G$ with bounded radius that intercepts all members of ${\cal F}$ (see Section \ref{sec:br-Helly}). In Section \ref{sec:cbc}, we generalize a known {\em characteristic cycle property} of chordal graphs (graphs with tree-length equal 1) and introduce two new cycle related parameters. We show that both these parameters coarsely define the tree-length.  
We conclude the paper with a few open questions (Section \ref{sec:concl}).

\medskip
\noindent
{\bf Basic notions and notations.} %\label{sec:notions}
All graphs $G$ in this paper are connected, unweighted, undirected, loopless and without multiple edges. % (unweighted simple graphs).
We assume also, for simplicity, that $G$ is finite although most of %many 
our non-algorithmic results and their proofs work for infinite graphs, too. 
For a (finite) graph $G=(V,E)$, we use $n$ and $|V|$ interchangeably to denote the number of vertices in $G$. Also, we use $m$ and $|E|$ to denote the number of edges. When we talk about two or more graphs, we may use $V(G)$ and $E(G)$ to indicate that these are the vertex and edge sets of a graph $G$. For a subset $S\subseteq V$, by $G[S]$ we  denote a subgraph of $G$ induced by the vertices of $S$.    

The {\em length of a path} $P(v,u):=(v=v_0,v_1,\dots,v_{\ell-1},v_{\ell}=u)$ from a vertex $v$ to a vertex $u$ is $\ell$, i.e., the number of edges in the path. The {\em distance} $d_G(u,v)$ between vertices $u$ and $v$ is the length of the shortest path connecting $u$ and $v$ in $G$. 
The distance between a vertex $v$ and a subset $S\subseteq V$ is defined as $d_G(v,S):=\min\{d_G(v,u): u\in S\}$. Similarly, let $d_G(S_1.S_2):=\min\{d_G(x,y): x\in S_1, y\in S_2\}$ for any two sets $S_1,S_2\subseteq V$. %$I(x,y)$ \\
The $k^{th}$ {\em power} $G^k$ of a graph $G$ is a graph that has the same set of vertices, but in which two distinct vertices are adjacent if and only if their distance in $G$ is at most $k$. 
The {\em disk} $D_r(s,G)$ of a graph $G$ centered at vertex $s \in V$ and with radius $r$ is the set of all vertices with distance no more than $r$ from $s$ (i.e., $D_r(s,G)=\{v\in V: d_G(v,s) \leq r \}$). We omit the graph name $G$ and write  $D_r(s)$ if the context is about only one graph. A cycle $C$ of a graph $G$ is called {\em geodesic} if $d_C(x,y)=d_G(x,y)$ for every $x,y$ in $C$. 

The \emph{diameter} of a subset $S\subseteq V$ of vertices of a graph $G$ is the largest distance  in $G$ between a pair of vertices of $S$, i.e., $\max_{u,v \in S}d_G(u,v)$. The \emph{inner diameter} of $S$  is the largest distance in $G[S]$ between a pair of vertices of $S$, i.e., $\max_{u,v \in S}d_{G[S]}(u,v)$.  The \emph{radius} of a subset $S\subseteq V$ of vertices of a graph $G$ is the minimum $r$ such that a vertex $v\in V$ exists with $S\subseteq D_r(v,G)$. The \emph{ inner radius} of a subset $S$ is the minimum $r$ such that a vertex $v\in S$ exists with $S\subseteq D_r(v,G[S])$.  

Let  $G=(V,E)$ be a graph and $X\subseteq V$ be a subset  of vertices of $G$. A disk $D_r(v)$ (a clique $C$) of $G$ is called a {\em balanced disk $($clique$)$ separator of $G$ with respect to $X$}, if every connected component of $G[V\setminus D_r(v)]$ (of $G[V\setminus C]$, respectively) has at most $|X|/2$ vertices from $X$.

Definitions of graph parameters considered in this paper, %measuring metric tree-likeness of a graph, 
as well as notions and notation local to a section, are given in appropriate sections. We realize that there are too many parameters and abbreviations and this may cause some difficulties in following them.  We give in Appendix a glossary for all parameters and summarize all inequalities and relations between them.

\section{Preliminaries on graph parameters} \label{sec:prelim-on-par}

We start with parameters central to our proofs and giving  the best to date approximation algorithm for computing the tree-length and the tree-breadth of a graph. 

\subsection{Layering partition, its cluster-diameter, cluster-radius, and a canonical tree} \label{sec:layer-partit}

Layering partition is a graph decomposition procedure introduced in~\cite{DBLP:journals/jal/BrandstadtCD99,DBLP:journals/ejc/ChepoiD00} and used in~\cite{BaInSi,DBLP:journals/jal/BrandstadtCD99,DBLP:journals/ejc/ChepoiD00,ChDrEsHaVaXi12,ChepoiDNRV12} %and~\cite{}
for embedding graph metrics into trees. 
%It provides a central tool in our investigation.

	\begin{figure}[htb]%[H]
 \footnotesize
		\centering
		\subfigure[][Layering of a graph $G$ with respect to $s$.]
		{
			\scalebox{0.24}[0.24]{\includegraphics{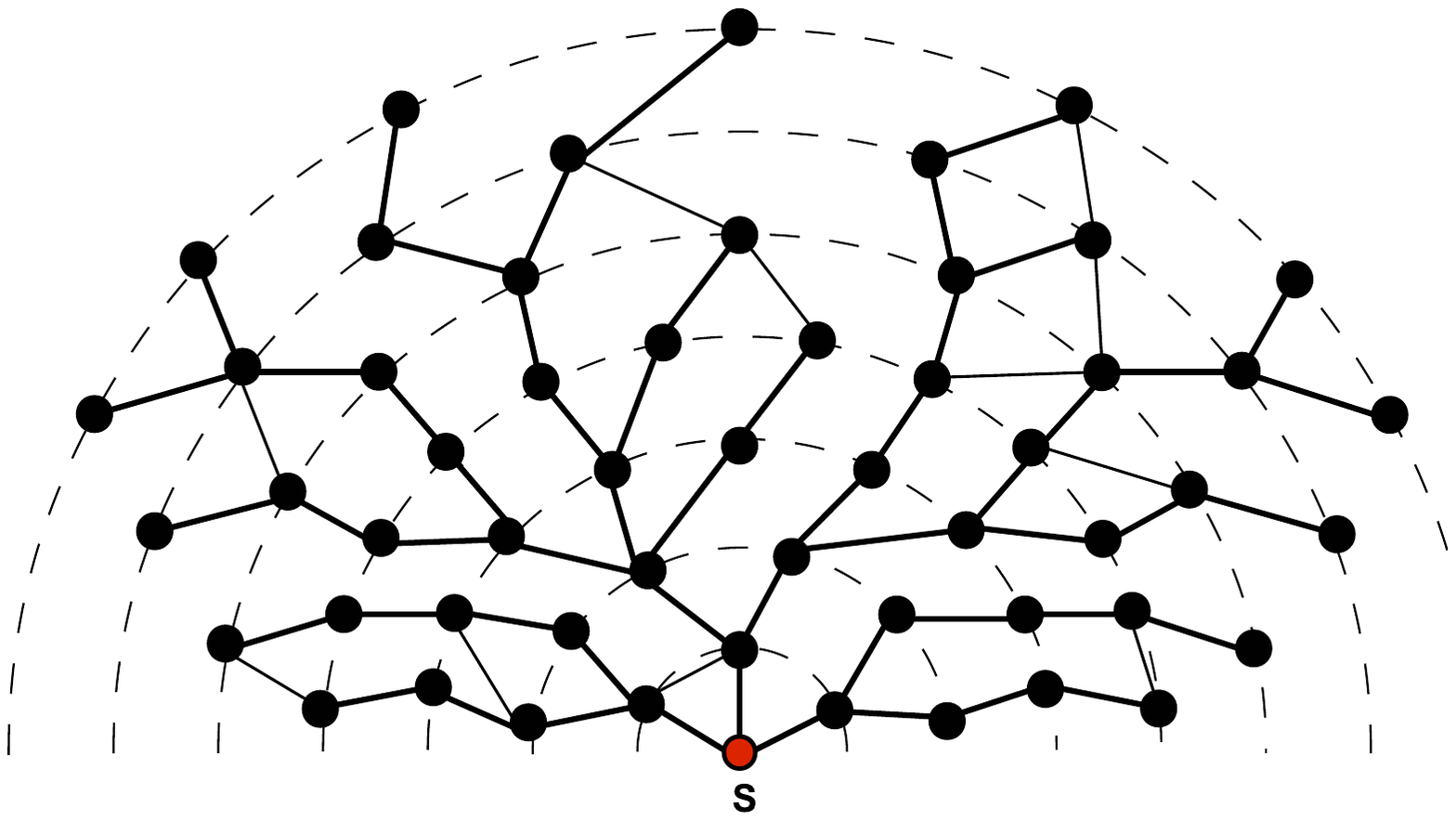}}
			\label{fig:layering}
		}
		\hspace{7ex}%
		\subfigure[][ Clusters of the layering partition $\mathcal{LP}(G,s)$.]
		{
			\scalebox{0.24}[0.24]{\includegraphics{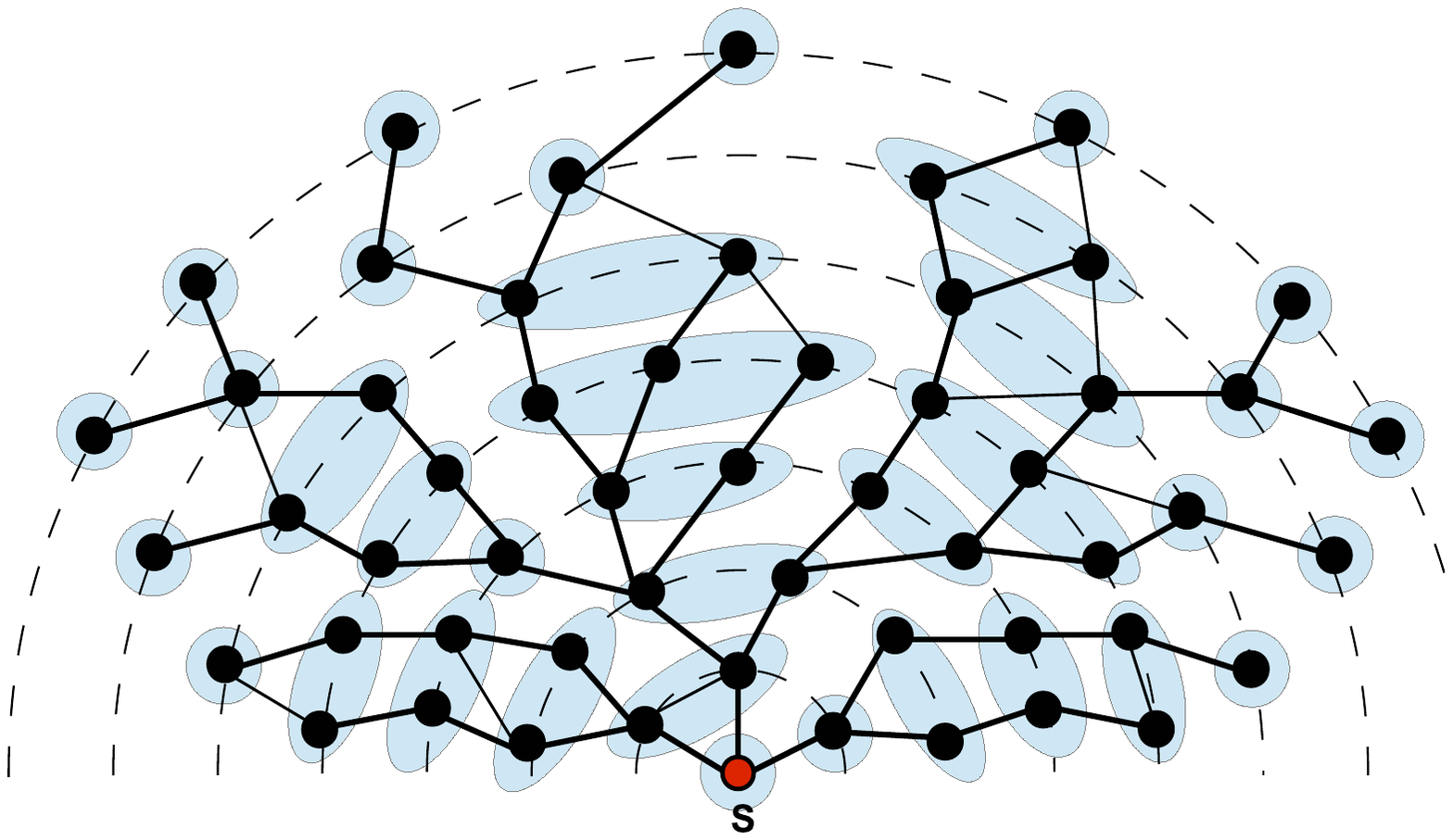}}
			\label{fig:Layering-clusters}
		}
		
		\subfigure[][Layering tree $\Gamma(G,s)$.]
		{
			\scalebox{0.24}[0.24]{\includegraphics{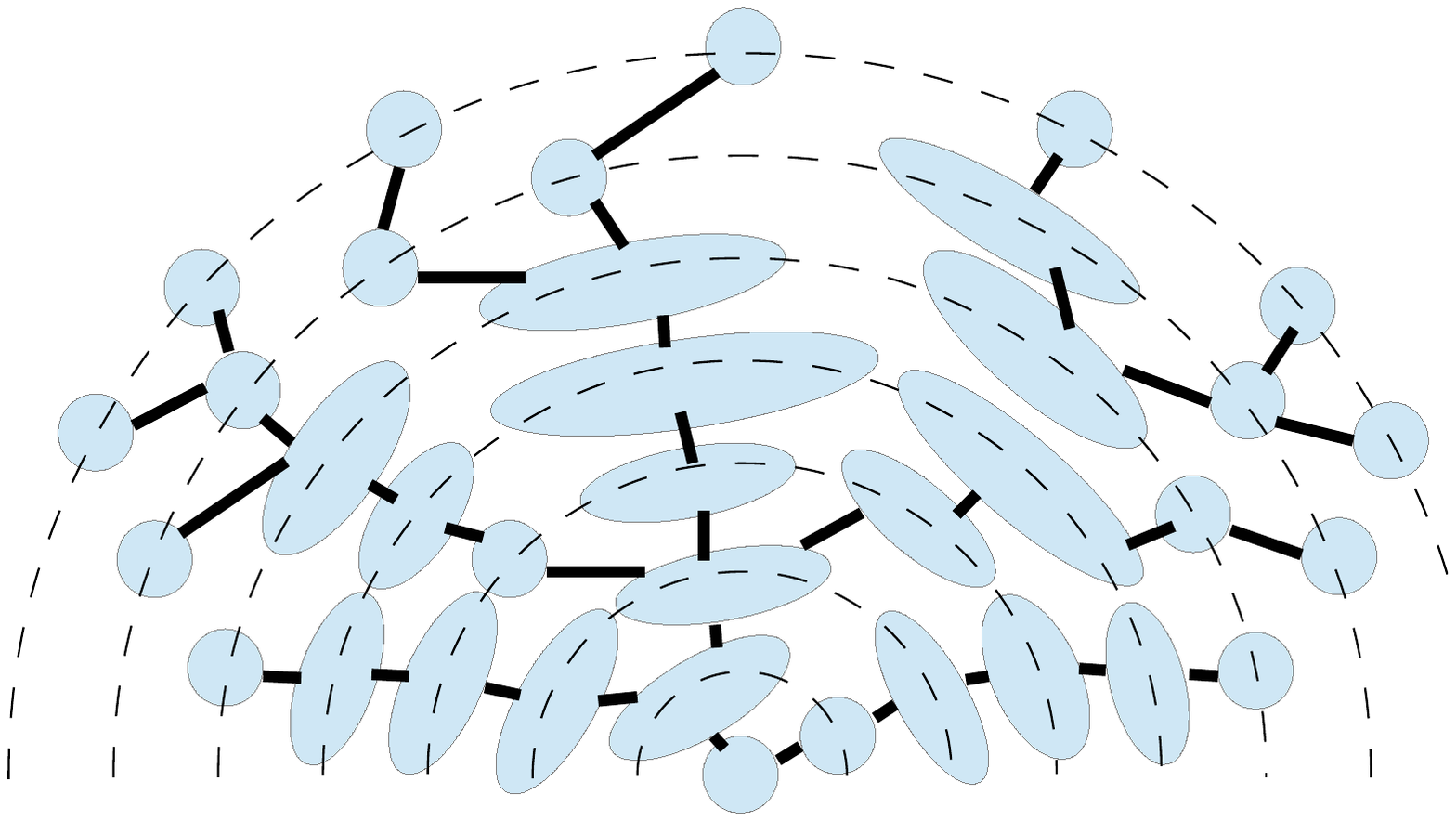}}
			\label{fig:gamma}
		}
        \hspace{7ex}%
		\subfigure[][ Canonical tree $H$. ]
		{
			\scalebox{0.24}[0.24]{\includegraphics{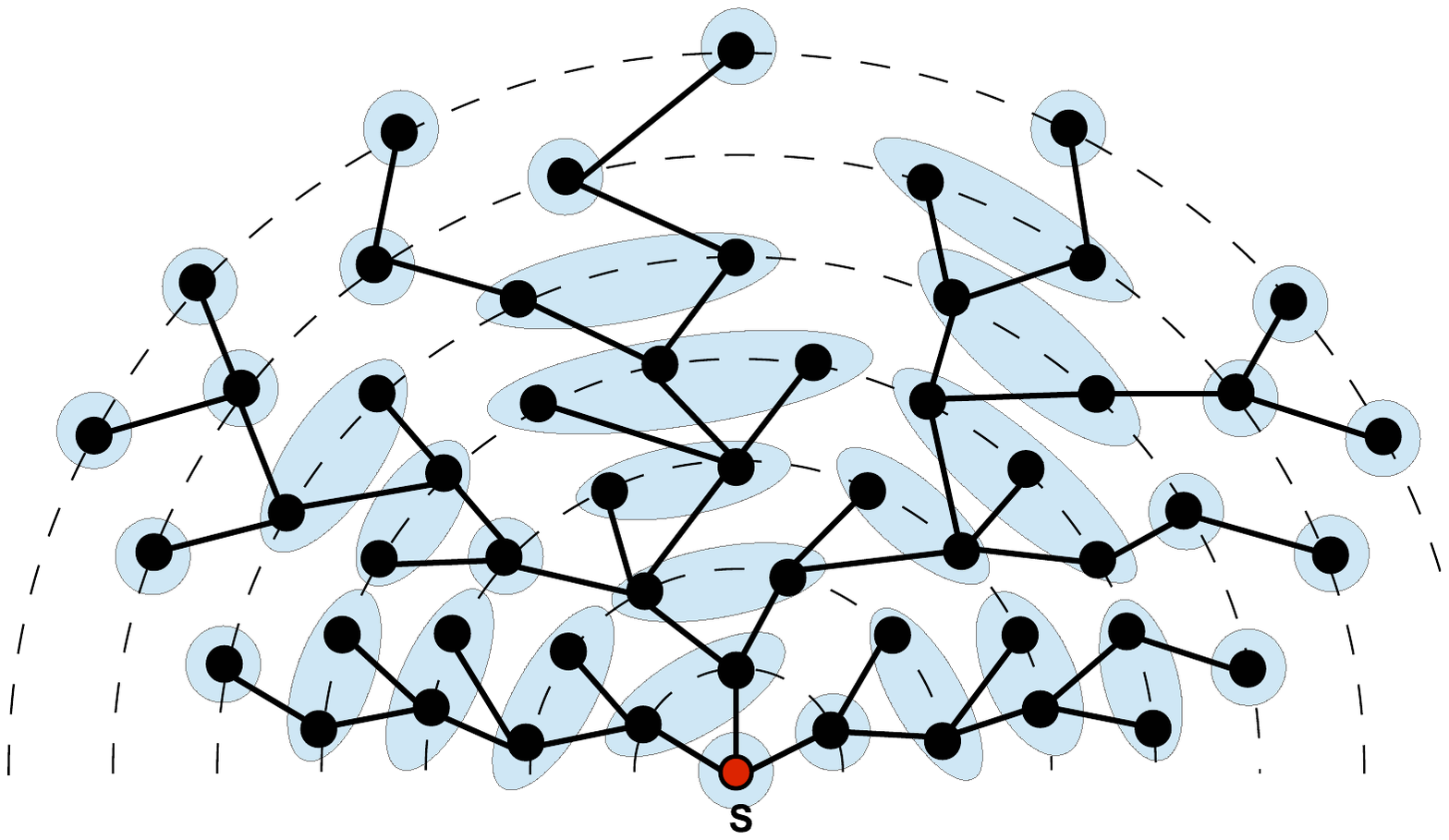}}
			\label{fig:tree-H}
		}
		%\vspace{2ex}
		\caption{\small Layering partition and associated constructs (taken from \cite{AbDr16}).}
		\label{fig:layering-partition}%\vspace{-2ex}
	\end{figure}

A \emph{layering} of a graph $G=(V, E)$ with respect to a start vertex $s$ is the decomposition of $V$ into $q+1$ layers (spheres) $L^i=\{u\in V:d_G(s,u)=i\},i=0,1,\dots,q$ (here, $q:=\max\{d_G(s,v): v\in V\}$). A \emph{layering partition} $\mathcal{LP}(G,s)=\{L^i_1,\ldots,L^i_{p_i}:i=0,1,\dots,q\}$ of $G$ is a partition of each layer $L^i$ into clusters $L^i_1,\dots,L^i_{p_i}$ such that two vertices $u,v \in L^i$ belong to the same cluster $L^i_j$ if and only if they can be connected by a path outside the disk $D_{i-1}(s)$ of radius $i-1$ centered at $s$. Here $p_i$ is the number of clusters in layer $i$. See Fig. \ref{fig:layering-partition} for an illustration. A layering partition of a graph can be constructed in $O(n+m)$ time (see~\cite{DBLP:journals/ejc/ChepoiD00}).

A \emph{layering tree} $\Gamma(G,s)$ of a graph $G$ with respect to a layering partition $\mathcal{LP}(G,s)$  is the graph whose nodes are the clusters of $\mathcal{LP}(G,s)$ and where two nodes $C=L_j^i$ and $C'=L_{j'}^{i'}$ are adjacent in $\Gamma(G,s)$ if and only if there exist a vertex $u \in C$ and a vertex $v\in C'$ such that $uv \in E$. It was shown in~\cite{DBLP:journals/jal/BrandstadtCD99} that the graph $\Gamma(G,s)$ is always a tree and, given a start vertex $s$, it can be constructed in $O(n+m)$ time~\cite{DBLP:journals/ejc/ChepoiD00}. Note that, for a fixed start vertex $s\in V$, the layering partition $\mathcal{LP}(G,s)$ of $G$ and its tree $\Gamma(G,s)$ are unique.

The \emph{cluster-diameter $\Delta_s(G)$ of layering partition $\mathcal{LP}(G,s)$ with respect to vertex $s$} is the largest diameter of a cluster in $\mathcal{LP}(G,s)$, i.e., $\Delta_s(G)=\max_{C \in \mathcal{LP}(G,s)} \max_{u,v\in C}d_G(u,v)$. Denote by $\Delta(G)$ ($\widehat{\Delta}(G)$) the minimum (the maximum, respectively) cluster-diameter over all layering partitions of $G$, i.e. $\Delta(G)=\min_{s \in V}\Delta_s(G)$ and $\widehat{\Delta}(G)=\max_{s \in V}\Delta_s(G)$.

The \emph{cluster-radius $\rho_s(G)$ of layering partition $\mathcal{LP}(G,s)$ with respect to a vertex $s$} is the smallest number $r$ such that for any cluster $C \in \mathcal{LP}(G,s)$ there is a vertex $v \in V$ with $C \subseteq D_r(v)$. Denote by $\rho(G)$ ($\widehat{\rho}(G)$)  the minimum (the maximum, respectively) cluster-radius over all layering partitions of $G$, i.e., $\rho(G)=\min_{s \in V}\rho_s(G)$ and $\widehat{\rho}(G)=\max_{s \in V}\rho_s(G)$. 

%Clearly, in view of tree $\Gamma(G,s)$ of $G$, the smaller parameters $\Delta_s(G)$ and $\rho_s(G)$, the closer graph $G$ is to a tree metrically.

Finding the cluster-diameter $\Delta_s(G)$ and the cluster-radius $\rho_s(G)$ for a given layering partition $\mathcal{LP}(G,s)$ of a graph $G$ requires $O(nm)$ time%\footnote{The parameters $\Delta(G)$, $\widehat{\Delta}(G)$  and $\rho(G)$, $\widehat{\rho}(G)$ can also be computed in total $O(nm)$ time for any graph $G$.}
, although the construction of the layering partition $\mathcal{LP}(G,s)$ itself, for a given vertex $s$, takes only $O(n+m)$ time. Since the diameter of any set is at least its radius and at most twice its radius, we have the following inequality: $$\rho_s(G) \leq \Delta_s(G) \leq 2\rho_s(G).$$
%$\Delta_s=\max\{d_G(x,y):x, y \mbox{ belong to the same cluster of } \mathcal{LP(s)}\}$

It is not hard to show that, for any graph $G$ and any two of its vertices $s,q$, $\Delta_q(G)\leq 3 \Delta_s(G)$. Thus, the choice of the start vertex for constructing a layering partition of $G$   is not that important.

\begin{proposition} [\cite{slimness}] \label{prop:ClustDiamAtAnys}
	Let $s$ be an arbitrary vertex of $G$. For every vertex $q$ of $G$, $\Delta_q(G)\le 3\Delta_s(G)$. In particular,
	$\Delta(G)\le \widehat{\Delta}(G)\le 3 \Delta(G)$ for every graph $G$.
\end{proposition}

%----------------

%Most of the graph parameters discussed in this paper could be related to a special tree $H$ introduced in~\cite{ChepoiDNRV12} and produced from a layering partition of a graph $G$.

%\textbf{Canonical tree} $\mathbf{H}$: 
A tree $H=(V,F)$ of a graph $G=(V,E)$, called a {\em canonical tree of $G$}, is constructed from a layering partition $\mathcal{LP}(G,s)$ of $G$ by identifying for each cluster $C=L^i_j \in \mathcal{LP}(G,s)$ an arbitrary vertex $x_C \in L_{i-1}$  which has a neighbor in $C = L^i_j$ and by making $x_C$ adjacent in $H$ with all vertices $v\in C$ (see Fig. \ref{fig:tree-H} for an illustration). The tree $H$ closely reproduces the global structure of the layering tree $\Gamma(G,s)$. 
%Vertex $x_C$ is called the support vertex for cluster $C= L^i_j$. 
It was shown in~\cite{ChepoiDNRV12} that the tree $H$ for a graph $G$ can be constructed in $O(n+m)$ time. 

The following result~\cite{ChepoiDNRV12} relates the cluster-diameter of a layering partition of $G$ to  embeddability of graph $G$ to the tree $H$.
\begin{proposition} [\cite{ChepoiDNRV12}]
\label{lem:cluster-diam}
For every graph $G=(V,E)$ and any vertex $s$ of $G$, $$\forall x,y \in V, ~~d_H(x,y)-2 \leq d_G(x,y) \leq d_H(x,y)+\Delta_s(G).$$
\end{proposition}

Proposition \ref{lem:cluster-diam} shows that the additive distortion of embedding of a graph $G$ to tree $H$
is bounded by $\Delta_s(G)$, the largest diameter of a cluster in a layering partition of $G$.
% ------------------ do we need this ?
%Using Lemma~\ref{lem:cluster-diam} and the previous inequality, we have:
%\[ d_H(x,y)-2 \leq d_G(x,y) \leq d_H(x,y)+2\rho_s \].
This result indicates that the smaller the cluster-diameter $\Delta_s(G)$ (cluster-radius $\rho_s(G)$) of $G$, the closer graph $G$ is to a tree metrically. Note that trees have cluster-diameter and cluster-radius equal to $0$. Results similar to Proposition \ref{lem:cluster-diam} were first used in~\cite{DBLP:journals/jal/BrandstadtCD99} to embed a chordal graph to a tree with an additive distortion of most 2 and  
in~\cite{DBLP:journals/ejc/ChepoiD00} to embed a $k$-chordal graph to a tree with an additive distortion at most $k/2 +2$. In~\cite{ChepoiDNRV12}, Proposition \ref{lem:cluster-diam} was used  to obtain a 6-approximation algorithm for the problem of optimal non-contractive embedding of an unweighted graph metric into a weighted tree metric. For every {\em chordal graph} $G$ (a graph whose largest induced cycles have length 3),  $\Delta_s(G) \leq 3$ and $\rho_s(G)\leq 2$ hold~\cite{DBLP:journals/jal/BrandstadtCD99}. For every {\em $k$-chordal graph} $G$ (a graph whose largest induced cycles have length $k$), $\Delta_s(G) \leq k/2 +2$ holds~\cite{DBLP:journals/ejc/ChepoiD00}. For every graph $G$ embeddable non-contractively into a (weighted) tree with multiplication distortion $\alpha$, $\Delta_s(G) \leq 3\alpha$ holds~\cite{ChepoiDNRV12}. See Section \ref{sec:embed-to-tree} for more on this topic.

\subsection{Tree-length and tree-breadth} \label{sec:tl}
It is known that the class of chordal graphs can be characterized in terms of existence of so-called {\em clique-trees}.
Let ${\cal C}(G)$ denote the family of  maximal (by inclusion) cliques of a graph $G$.
A {\em clique-tree} ${\cal CT}(G)$ of $G$ has the maximal cliques of $G$ as its
nodes, and for every vertex $v$ of $G$, the maximal cliques containing $v$ form a subtree of ${\cal CT}(G)$.
%The existence of a clique tree characterizes chordal graphs:

\begin{proposition} %[Buneman \cite{Bunem1974}, Gavril \cite{Gavri1974}, Walter \cite{Walte1972}]
 [\cite{Bunem1974,Gavri1974,Walte1972}]\label{cliquetreechordalgr} 
A graph is chordal if and only if it has a clique-tree.
\end{proposition}

In their work on graph minors \cite{RobSey1986}, Robertson and Seymour introduced the notion of tree-decomposition which generalizes the notion of clique-tree.
A {\em tree-decomposition} of a graph $G$ is a tree $\cT(G)$  whose nodes, called {\em bags}, are subsets of $V(G)$ such that:

\begin{enumerate}
\item[(1)] $\bigcup_{B\in V(\cT(G))} B = V(G)$,

\item[(2)] for each edge $vw\in E(G)$, there is a bag $B\in V(\cT(G))$
such that $v,w\in B$, and

\item[(3)] for each $v\in V(G)$ the set of bags  $\{ B: B\in V(\cT(G)), v\in B\}$ forms a subtree  $\cT_v(G)$  of  $\cT(G)$.
\end{enumerate}

Tree-decompositions were used in defining several graph parameters.
The {\em tree-width} of a graph $G$ is defined as minimum of $\max_{B\in V(\cT(G))} |B| - 1$ over all tree-decompositions $\cT(G)$ of $G$ and is denoted by $\tw(G)$ \cite{RobSey1986}. The {\em length} of a tree-decomposition $\cT(G)$ of a graph $G$ is $\max_{B\in V(\cT(G))}\max_{u,v\in B}d_G(u,v)$, and the {\em tree-length} of $G$, denoted by $\tl(G)$, is the minimum of the length, over all tree-decompositions of $G$ \cite{Dorisb2007}. These two graph parameters generally are not related to each other. For instance, cliques (or, generally, all chordal graphs) have unbounded tree-width and tree-length 1, whereas cycles have tree-width 2 and unbounded tree-length. However, in some special cases they are coarsely equivalent. It is known that the tree-length $\tl(G)$ of any graph $G$ is at most $\lfloor\ell(G)/2\rfloor$ times its tree-width $\tw(G)$ (where $\ell(G)$ is the size of a largest geodesic cycle in $G$) \cite{CDN2016,Conn-tw}  and that for any graph $G$ that excludes an apex graph $H$ as a minor, $\tw(G)\le c_H\cdot\tl(G)$ for some constant $c_H$ only depending on $H$ \cite{CDN2016}. 
The tree-length of any graph $G$ can equivalently be defined in two other ways. As it is shown in \cite{b-length}, tree-length  
equals branch-length (see  \cite{b-length} for more details). Tree-length equals also short fill-in to a chordal graph, that is,  the smallest number $\ell\in N$ that permits an edge set $E'$ between vertices of $G$ such
that $G' = (V, E \cup E')$ is a chordal graph and for all $uv \in E'$, $d_G(u, v) \le\ell$ holds (see \cite{dagstuhl} and \cite{fid-param}). 
%Interestingly, the tree-length of a graph can be approximated in polynomial time within a
%constant factor \cite{DoGa2007} whereas such an approximation factor is unknown for the tree-width.
%
%For the purpose of this paper, we introduce yet another graph parameter based on the notion of tree-decomposition. It is very similar to the notion of tree-length but is more appropriate for our discussions, and moreover it will lead to a better constant in our approximation ratio presented in Section \ref{subsec:back} for the {\sc tree $t$-spanner} problem on general graphs.
%
The {\em  breadth}  of a tree-decomposition $\cT(G)$ of a graph $G$ is the minimum integer $k$ such that for every $B\in V(\cT(G))$ there is a vertex $v_B\in V(G)$ with $B\subseteq D_k(v_B,G)$ (i.e., each bag $B$ has radius at most $k$ in $G$). Note that vertex $v_B$ does not need to belong to $B$. The {\em tree-breadth}  of $G$, denoted by $\tb(G)$, is the minimum of the breadth, over all tree-decompositions of $G$ \cite{tree-spanner-appr}. 
Note also that in \cite{acyclic-clustering} independently a notion similar to tree-length (and tree-breadth) was introduced for purposes of compact and efficient routing in certain graph classes. It was called $(R,D)$-acyclic clustering. An $(R,2R)$-acyclic clustering is exactly a tree-decomposition with breadth $R$ and  a $(D,D)$-acyclic clustering is exactly a tree-decomposition with length $D$.  

Evidently, for any graph $G$, $1\leq \tb(G)\leq \tl(G)\leq 2 \tb(G)$ holds. Hence, if one parameter is bounded by a constant for a graph $G$ then the other parameter is bounded for $G$ as well. We say that a family of graphs ${\cal G}$ is {\em of bounded tree-length} (equivalently, {\em of bounded tree-breadth}), if there is a constant $c$ such that for each graph  $G$ from ${\cal G}$,  $\tl(G)\leq c$. 
It is known that checking whether a graph $G$ satisfies 
$\tl(G)\le \lambda$ or $\tb(G)\le r$ is NP-complete for each $\lambda >1$ \cite{NPc-tl} and each $r\ge 1$ \cite{NPc-tb}. Furthermore,  unless $P= NP$, there is no polynomial time algorithm to calculate a tree-decomposition, for a given graph $G$, of length smaller than $\frac{3}{2}\tl(G)$ \cite{NPc-tl}; and for any $\epsilon>0$, it is NP-hard to approximate the tree-breadth of a given graph by a factor of $2 -\epsilon$ \cite{NPc-tb}. 

Notice that in the definition of the length of a tree-decomposition, the distance between vertices of a bag is measured in the entire graph $G$. If the distance between any vertices $x,y$ from a bag $B$ is measured in $G[B]$, then one gets the notion of the inner length of a tree-decomposition \cite{BerSey2024}. The {\em inner length} of a tree-decomposition $\cT(G)$ of a graph $G$ is $\max_{B\in V(\cT(G))}\max_{u,v\in B}d_{G[B]}(u,v)$, and the {\em inner tree-length} of $G$, denoted by $\itl(G)$, is the minimum of the inner length, over all tree-decompositions of $G$.  Similarly, the {\em  inner breadth}  of a tree-decomposition $\cT(G)$ of a graph $G$ is the minimum integer $k$ such that for every $B\in V(\cT(G))$ there is a vertex $v_B\in B$ with $d_{G[B]}(u,v_B)\le k$ for all $u\in B$ (i.e., each subgraph $G[B]$ has radius at most $k$). The {\em inner tree-breadth}  of $G$, denoted by $\itb(G)$, is the minimum of the inner breadth, over all tree-decompositions of $G$ \cite{Diestel++}.  

Interestingly, the inner tree-length (inner tree-breadth) of $G$ is at most twice the tree-length (tree-breadth, respectively) of $G.$ 

\begin{proposition} [\cite{Diestel++,BerSey2024}] \label{prop:inner}
	For every graph $G$,
	$\tl(G) \leq \itl(G)\le 2\cdot\tl(G)$ and $\tb(G) \leq \itb(G)\le 2\cdot\tb(G).$ 
\end{proposition}

Since $\itl(G)$ and $\itb(G)$ are not that far from $\tl(G)$ and $\tb(G)$, in what follows, we will work only with $\tl(G)$ and $\tb(G)$. 

The following proposition establishes a relationship between the tree-length and the cluster-diameter of a layering partition of a graph.
\begin{proposition} [\cite{Dorisb2007}] \label{prop:dorisb}
	For every graph $G$ and any vertex $s$,
	$\frac{\Delta_s(G)}{3} \leq \tl(G) \leq \Delta_s(G)+1.$ In particular, $\frac{\widehat{\Delta}(G)}{3} \leq \tl(G) \leq \Delta(G)+1$ for every graph $G$.
\end{proposition}

Thus, the cluster-diameter $\Delta_s(G)$ of a layering partition provides easily computable bounds for the hard to compute parameter $\tl(G)$.

Similar inequalities are known for $\rho_s(G)$ and $\tb(G)$. 
\begin{proposition} [\cite{AbDr16,DDGY-spanners,tree-spanner-appr}] \label{prop:Muad-Feodor}
	For every graph $G$ and any vertex $s$,
	$\frac{\rho_s(G)}{3} \leq \tb(G) \leq \rho_s(G)+1$ and $\rho_s(G)\le 2\cdot\tl(G)$.
\end{proposition}

For a given graph  $G$ and its arbitrary vertex $s$, the layering tree $\Gamma(G,s)$, obtained in linear time from $G$ (see Section \ref{sec:layer-partit}), is almost a tree-decomposition of $G$. It only violates the condition (2) of a tree-decomposition. This tree  $\Gamma(G,s)$ can be transformed into a tree-decomposition by expanding its clusters as follows. For a cluster $C$, add all vertices from the parent of $C$ in  $\Gamma(G,s)$ which are adjacent to a vertex in $C$. That is, for each cluster $C \subseteq L^i$, create a bag $B_C = C\cup (N_G(C) \cap L^{i-1})$, where $N_G(C)$ denotes all vertices of $G$ that are adjacent to vertices of $C$. As it was shown in \cite{AbDr16,Dorisb2007,tree-spanner-appr}, the obtained tree-decomposition has length at most $3\cdot\tl(G)$\footnote{In \cite{Dorisb2007}, it was claimed to have the length at most $3\cdot\tl(G)+1$, but as mentioned in \cite{AbDr16}, its actual length is at most $3\cdot\tl(G)$.}  and breadth at most $3\cdot\tb(G)$. Hence, 3-approximations of $\tl(G)$ and of $\tb(G)$ and a corresponding tree-decomposition of length at most $3\cdot\tl(G)$ and of breadth at most $3\cdot\tb(G)$ can be computed  in linear time. This is the best to date approximation algorithm for computing the tree-length and the tree-breadth of a graph.

\subsection{Embedding a graph into a tree}\label{sec:embed-to-tree}

Several different types of embeddings of (unweighted) graphs into trees were considered in literature (see \cite{AbDr16,Farach,BaDeHaSiZa08,BaInSi,BeRa22,BerSey2024,DBLP:journals/jal/BrandstadtCD99,BrDrchapter,CaiDerek,DBLP:journals/ejc/ChepoiD00,ChDrEsHaVa08,ChepoiDNRV12,ChepoiFichet,WG13-Dragan,tree-spanner-appr,ApprTree-DRYan,EmekPeleg,Kerr}). We will elaborate only on those results that are relevant to our discussion. In \cite{{BaDeHaSiZa08,BaInSi,ChepoiDNRV12}}, unweighted graphs embeddable non-contractively into a (weighted) tree with multiplicative distortion $\alpha$ are considered. For each such graph $G=(V,E)$ there is a  (weighted) tree $T=(V',E')$ with $V'\subseteq V$ such that $d_G(x,y)\le d_T(x,y)\le \alpha\cdot d_G(x,y)$ for every $x,y\in V$. A {\em tree distortion} $\td(G)$ (see \cite{AbDr16,WG13-Dragan}) of a graph $G$ is the minimum $\alpha$ such that $G$ can be embedded non-contractively into a (weighted) tree with multiplicative distortion $\alpha$. In \cite{ChepoiDNRV12,tree-spanner-appr} (see also \cite{AbDr16,WG13-Dragan}), it is shown that for every graph $G$, tree-distortion parameter $\td(G)$ is coarsely equivalent to $\Delta_s(G)$ (for any $s\in V$) and to $\tl(G)$. 

\begin{proposition} [\cite{AbDr16,ChepoiDNRV12,tree-spanner-appr}] \label{prop:td-tl}
	For every graph $G$ and any vertex $s$,
	$$\frac{\Delta_s(G)}{3} \leq \tl(G) \leq \td(G) \le 2\cdot \Delta_s(G)+2.$$ 
 In particular,
	$\frac{\widehat{\Delta}(G)}{3} \leq \tl(G) \leq \td(G) \le 2\cdot \Delta(G)+2$ for every graph $G$.
\end{proposition}
%{\color{red}/* check old computer for proof $\td\le \tl$.   */}

It is known that deciding whether $\td(G)\le \alpha$ is NP-complete, and even more, the hardness result of~\cite{Farach}
implies that it is NP-hard to $\gamma$-approximate $\td(G)$ for some small constant $\gamma$. 
As we have mentioned earlier, in~\cite{ChepoiDNRV12}, Proposition \ref{lem:cluster-diam} is used  to obtain a (best to date) 6-approximation algorithm for the problem of optimal non-contractive embedding of an unweighted graph metric into a weighted tree metric. 
In fact, one of the results of~\cite{ChepoiDNRV12} says that if for a graph $G=(V,E)$ there is a (weighted) tree $T=(V',E')$ with $V'\subseteq V$ such that $d_G(x,y)\le d_T(x,y)\le \alpha\cdot d_G(x,y)$ for every $x,y\in V$, then for any $s\in V$, the cluster-diameter $\Delta_s(G)$ of layering partition $\mathcal{LP}(G,s)$ of $G$ is at most $3\alpha$ and, hence, a canonical tree $H$ (see Section \ref{sec:layer-partit}) of $G$ satisfies 
$d_H(x,y)-2 \leq d_G(x,y) \leq d_H(x,y)+3\alpha$ for all $x,y \in V$.
It turns out that, for any unweighted graph $G$, it is possible to turn its non-contractive multiplicative low-distortion embedding into a weighted tree to an additive low-distortion embedding to an unweighted tree with the same vertex set as in $G$. Later such a phenomenon of turning multiplicative distortions to additive distortions were observed for other types of embedding into trees \cite{BeRa22,BerSey2024,Kerr} (see below). 

As it was noticed in~\cite{ChepoiDNRV12}, a more general result can be stated: if for a graph $G=(V,E)$ there is a (weighted) tree $T=(V',E')$ with $V'\subseteq V$ such that $d_T(x,y)\le \alpha\cdot d_G(x,y)+\beta$ and $d_G(x,y)\le \lambda\cdot d_T(x,y)+\delta$ for every $x,y\in V$, then for any $s\in V$, the cluster-diameter $\Delta_s(G)$ of layering partition $\mathcal{LP}(G,s)$ of $G$ is at most $3(\alpha(\lambda+\delta)+\beta)$ and, hence, a canonical tree $H$ of $G$ satisfies 
$d_H(x,y)-2 \leq d_G(x,y) \leq d_H(x,y)+ 3(\alpha(\lambda+\delta)+\beta)$ for all $x,y \in V$ (getting again only an additive distortion; compare with a result of Kerr below and Proposition \ref{prop:BerSey2024}). 

In \cite{CaiDerek,tree-spanner-appr,EmekPeleg}, unweighted graphs embeddable with low distortion to their spanning trees are considered. A spanning tree $T$ of a graph $G$ is called a {\em tree $t$-spanner} of $G$ if $d_T(x,y)\le t\cdot d_G(x,y)$ holds for every $x,y\in V$.  In \cite{tree-spanner-appr}, it is shown that any graph admitting a tree $t$-spanner has tree-breadth at most $\lceil t/2\rceil$. Furthermore, every graph $G$ with $\tb(G)\le \delta$ admits a tree $O(\delta\log n)$-spanner\footnote{Note that the $\log n$ factor here cannot be dropped. There are graphs with tree-length (and, hence, tree-breadth) one (e.g., chordal graphs) which have tree $t$-spanner only for $t=\Omega(\log n)$~\cite{BerSey2024,tree-spanner-appr,add-spanner}.} constructible in $O(m\log n)$ time. The latter provides an efficient $O(\log n)$-approximation algorithm for the problem of finding a tree $t$-spanner with minimum $t$ of a given graph $G$. Another efficient $O(\log n)$-approximation algorithm can be found in \cite{EmekPeleg}. Interestingly, as a recent paper \cite{BeRa22} demonstrates, for every graph $G$ admitting a tree $t$-spanner, one can efficiently construct a spanning tree $T$ such that $d_T(x,y)\le d_G(x,y)+O(t\log n)$ holds for every $x,y\in V$ (turning a multiplicative distortion into an additive distortion with an additional logarithmic factor). Note also that deciding whether a given graph $G$ admits a tree $t$-spanner  is an NP-complete problem even for $G$ being a chordal graph, i.e., a graph with tree-length (and tree-breadth) equal one, and for every fixed $t>3$ (see \cite{BrDrLeLe}), while it is polynomial time solvable for all graphs of bounded tree-width (see \cite{MakowskyRotics??}). 

A most general notion of embedding into trees is given by a ``quasi-isometry" from a graph to a tree. This is the concept from metric spaces, but we define it just for graphs. Let $G$ be a graph, $T$ be a tree (possibly weighted) and $\psi: V(G)\rightarrow  V(T)$ be a map. Let $L\ge 1$ and $C\ge 0$. We say that $\psi$ is an $(L,C)$-quasi-isometry if: 

\begin{enumerate}
    \item[(i)] for all $u,v\in V(G)$, $\frac{1}{L}d_G(u,v)-C\le d_T(\psi(u),\psi(v))\le L d_G(u,v)+C$; and  \\
    \item[(ii)] for every $y\in V(T)$ there is $v\in V(G)$ such that $d_T(\psi(v),y)\le C$.  
\end{enumerate}

Although $(L,C)$-quasi-isometry from $G$ to $T$ looks very general, a recent result of Kerr \cite{Kerr} states that for all $L,C$ there is a $C'$ such that if there is an $(L,C)$-quasi-isometry from $G$ to a tree, then there is a $(1,C')$-quasi-isometry from $G$ to a tree. Kerr's proof was for metric spaces. For graphs, this result was extended in \cite{BerSey2024} in  the following way. 

\begin{proposition} [\cite{BerSey2024}] \label{prop:BerSey2024}
For every graph $G$, the following three statements are equivalent:
\begin{enumerate}
    \item[(i)] the tree-length of $G$ is bounded;
    \item[(ii)] there is an $(L,C)$-quasi-isometry to a tree with $L,C$ bounded;
    \item[(iii)] there is an $(1,C')$-quasi-isometry to a tree with $C'$ bounded. 
\end{enumerate}
\end{proposition}

Because of equivalency of $(ii)$ and $(iii)$, \cite{BerSey2024} introduced a new graph parameter $\ad(G)$. The additive distortion $\ad(G)$ of a graph $G$ is the minimum $k$ such that there is a (weighted)  tree $T$ and a map $\psi: V(G)\rightarrow  V(T)$ such that $$|d_G(u,v)-d_T(\psi(u),\psi(v))|\le k$$ holds for every $u,v\in V(G)$. It was shown \cite[Theorem 4.1, Theorem 4.2]{BerSey2024} that the following inequalities between $\tl(G)$ and $\ad(G)$  hold. 

\begin{proposition} [\cite{BerSey2024}] \label{prop:ad-tl-Seymour}
	For every graph $G$, $\frac{\tl(G)-2}{2}\leq \ad(G)\leq 6\cdot \tl(G).$
\end{proposition}

Notice that in the definition of $\ad(G)$, $V(G)$ is not necessarily equal to $V(T)$, and $T$ can have weights. Motivated by the ability of a canonical tree $H$ obtained from a layering partition to additively approximate graph distances (see Section  \ref{sec:layer-partit}), here, we restrict further the parameter $\ad(G)$ and recall a notion of {\em distance $k$-approximating trees}, introduced in \cite{DBLP:journals/jal/BrandstadtCD99,DBLP:journals/ejc/ChepoiD00} and further investigate in \cite{AbDr16,ChDrEsHaVaXi12,ChDrEsHaVa08,WG13-Dragan,ApprTree-DRYan}.  A (unweighted) tree $T=(V,E')$ is a {\em distance $k$-approximating tree} of a graph $G=(V,E)$ if $$|d_G(u,v)-d_T(u,v)|\le k$$ holds for every $u,v\in V$. Denote by $\adt(G)$ the minimum $k$ such that $G$ has a distance $k$-approximating tree.  
It is known that every chordal graph has a distance 2-approximating tree~\cite{DBLP:journals/jal/BrandstadtCD99}, every $k$-chordal graph has a distance $(k/2 +2)$-approximating tree~\cite{DBLP:journals/ejc/ChepoiD00}, every graph with tree-length $\lambda$ has  a distance $3\lambda$-approximating tree \cite{AbDr16,WG13-Dragan}, and every $\delta$-hyperbolic graph has a  distance $O(\delta\log n)$-approximating tree \cite{ChDrEsHaVa08}. See also \cite{ApprTree-DRYan} for some hardness results. %{\color{red}/*  ??? complexity  */}

In Section \ref{sec:proofs} (see Theorem \ref{th:adt-tl-delta} and Corollary \ref{cor:ineq-tl-adt}), we will show 
$$\adt(G)\leq \Delta(G)\le \Delta_s(G)\le\widehat{\Delta}(G)\leq 3\cdot \tl(G)\le 6\cdot \adt(G)+3,$$ 
$$\frac{\tl(G)-1}{2}\leq \adt(G)\leq 3\cdot \tl(G).$$
Our embedding is a most restrictive one as it requires a low additive distortion embedding to an unweighted tree with the same vertex set as in $G$. 
%{\color{red}/* ours is a more restricted variant.   */}

\subsection{Bottleneck constant} \label{sec:early-bn}
There were results already known that characterize when a graph is quasi-isometric
to a tree. The {\em bottleneck constant} of a graph $G$ is the least integer $r$ such that if $P(u,v)$ is a shortest path of $G$ between $u, v$, of even length and with middle vertex $w$, then every path between $u, v$ contains a vertex that has distance at most $r$ from $w$. A theorem of Manning for geodesic metric spaces implies the following.

\begin{proposition}[\cite{manning}] \label{th:Manning} 
For all $L\ge 1$ and $C\ge 0$, there exists $r$ such that, for all graphs $G$,
if there is an $(L,C)$-quasi-isometry from $G$ to a tree, then $G$ has bottleneck constant at most $r$. Conversely, for all $r$ there exist $L\ge 1$ and $C\ge 0$ such that, for all graphs $G$, if $G$ has bottleneck constant at most $r$, then there is an $(L,C)$-quasi-isometry from $G$ to a tree. 
\end{proposition} 
A similar result with an additional characterization through so-called {\em fat $K_3$-minors} is given also in \cite[Theorem 3.1]{GeorPapa2023} (see below). 

From Proposition \ref{prop:BerSey2024}, it already follows that if the bottleneck constant  of a graph $G$ is bounded then the tree-length of $G$ is also bounded. A new graph-theoretic proof from \cite{BerSey2024} has more explicit control over the bounds. Let $\bn(G)$ denote the bottleneck constant  of $G$. 

\begin{proposition}[\cite{BerSey2024}] \label{pr:Seymour--Manning} 
For every graph $G$,   $\frac{2}{3}\bn(G)\leq \tl(G)\leq 4\cdot \bn(G)+3$ and  
$\ad(G)\le 24\cdot \bn(G)+18$. 	
\end{proposition} 

In this paper, we give simpler (more direct) proofs for those bounds by employing a layering partition and its cluster-diameter.  See Lemma  \ref{lm:BNC_ClusterDiam}, Lemma  \ref{lm:ClusterDiam_bnc}, 
Theorem \ref{th:bnc-delta-tl} and 
Corollary \ref{cor:ineq-tl-bnc}. We get, in fact, $\adt(G)\le 4\cdot \bn(G)+2$ (see Theorem \ref{th:adt-tl-delta}). 	
\medskip

Recently, in \cite{GeorPapa2023}, the Manning's Theorem was  extended to include also a characterization via $K$-fat $K_3$-minors. Although a $K$-fat $H$-minor can be defined with any graph $H$ \cite{GeorPapa2023}, here we give a definition only for $H=K_3$ as we work in this paper only with this minor.  % in the following way. 
It is said that a graph $G$ has a {\em $K$-fat $K_3$-minor} ($K>0$) if there are three connected subgraphs $H_1$, $H_2$, $H_3$ and three simple paths $P_{1,2}$, $P_{2,3}$, $P_{1,3}$ in $G$ such that for each $i,j\in \{1,2,3\}$ ($i\neq j$), $P_{i,j}$ has one end in $H_i$ and the other end in $H_j$ and  $|P_{i,j}\cap V(H_i)|=|P_{i,j}\cap V(H_j)|=1$, and
\begin{itemize}
%    \item{} [conditions on being a $K_3$-minor]:  
%    \item[~~~~~~] 
    
    \item{} [conditions for being $K$-fat:]  
%   \item[~~~~~~] 
    $d_G(V(H_i),V(H_j))\ge K$,  $d_G(P_{i,j},V(H_k))\ge K$ ($k\in \{1,2,3\}, k\neq i, j$) and the distance between any two paths $P_{1,2}$, $P_{2,3}$, $P_{1,3}$ is at least $K$.
\end{itemize}

%$d_G(V(H_i),V(H_j))\ge K$  and  $d_G(P_{1,2} $ 

\begin{proposition} [\cite{GeorPapa2023}] \label{prop:Papaaoglu} The following are equivalent for every graph $G$: 
\begin{itemize}
    \item[(i)] The bottleneck constant $\bn(G)$ of $G$ is bounded by a constant; 
  \item[(ii)]    $G$ has no $K$-fat $K_3$-minor for some constant $K>0$; and
 \item[(iii)] $G$ is $(1,C)$-quasi-isometric to a tree for some constant $C$. 
\end{itemize}
\end{proposition} 
They show in \cite{GeorPapa2023} that $K$ can be chosen to be greater than $2\cdot\bn(G)+1$ in their proof of $(i)\Rightarrow (ii)$ and $C$ can be chosen to be at most $14K$ in their proof of $(ii)\Rightarrow (iii)$. 

In this paper, we give an alternative simple proof of analog of Proposition \ref{prop:Papaaoglu} which also improves the constant $C$ in $(iii)$. %simpler (more 
%direct) proofs for those bounds by employing a layering partition %and  its cluster-diameter.   

\subsection{McCarty-width}\label{sec:early-mcw}
A graph $G$ has {\em McCarty-width} $r$ if $r\ge 0$ is minimum such that the following holds: for every three vertices $u, v, w$ of $G$, there is a vertex $x$ such that no connected component of $G[V\setminus D_r(x)]$  contains two of $u, v, w$. Let $\mc(G)$ denote the McCarty-width of $G$. Rose McCarty suggested that $\tl(G)$ is small if and only if $\mc(G)$ is small. This was proved in \cite{BerSey2024}.

\begin{proposition} [\cite{BerSey2024}] \label{prop:sey-mcw} For every graph $G$, $\frac{\tl(G)- 3}{6} \le \mc(G) \le \tl(G).$
\end{proposition} 

We also consider in this paper a more general parameter related to a small radius balanced disk separator. Balanced disk separators proved to be very useful in designing efficient approximation algorithms for the problem of finding a tree $t$-spanner with minimum $t$ of a given graph $G$ \cite{tree-spanner-appr}  and in constructing collective additive tree spanners (see, e.g.,  \cite{CTS1,CTS2,CTS3}) as well as compact distance and routing labeling schemes for variety of graph classes (see, e.g., \cite{CTS2,CTS3,GKKPP2001,KKP2000}).  Experiments performed in \cite{MSstudent} show that many real-life networks have small radius balanced disk separators.

Let  $G=(V,E)$ be a graph and $X\subseteq V$ be a subset  of vertices of $G$. Recall that a disk $D_r(v)$ of $G$ is called a {\em balanced disk separator of $G$ with respect to $X$}, if every connected component of $G[V\setminus D_r(v)]$ has at most $|X|/2$ vertices from $X$. 
Let {\em McCarty-width of order $k$} ($k\ge 3$) of $G$ (denoted by $\mc_k(G)$) be the minimum $r\ge 0$ such that for every subset $X\subseteq V$ with $|X|=k$, there is a vertex $v$ in $G$ such that $D_r(v)$ is a balanced disk separator  of $G$ with respect to $X$.  Clearly, $\mc(G)=\mc_3(G)$. 

In this paper, we show that for every graph $G$, every vertex $s$ of $G$, and every integer $k\ge 3$, $\Delta_s(G)\leq 6\cdot \mc(G)$ (see Lemma  \ref{lm:ClusterDiam_mcw}), $\mc_k(G)\le \rho_s(G)\le \Delta_s(G)$ (see Lemma \ref{lm:Clusterrad_BDS}) and $\mc_k(G)\le \tb(G)$ (see Lemma \ref{lm:tb_mcw_k}).  Furthermore, for any subset $X\subseteq V$ of vertices of $G$, a balanced disk separator $D_r(u)$ with $r\le \Delta_s(G)$ can be found in linear time and a balanced disk separator $D_r(u)$ with minimum $r$ (hence, with $r\le \tb(G)$) can be found in $O(nm)$ time. Consequently, we slightly improve the bounds in Proposition  \ref{prop:sey-mcw} by getting   $\frac{\tl(G)- 1}{6} \le \mc(G) \le \tb(G).$ Our proofs are  simpler and more direct and again employ a layering partition and its cluster-radius. 
 
\subsection{Geodesic loaded cycles} \label{sec:glc}
In a quest to find a property of cycles that guarantees a bound on the tree-length (i.e., a cycle property that coarsely describes the tree-length), Berger and Seymour \cite{BerSey2024} introduced a new notion of bounded load geodesic cycles. 

Let $C$ be a cycle of $G$ and let $F\subseteq E(C)$. The pair $(C, F )$ is called a {\em loaded cycle} of $G$, and $|F|$ is called its {\em load}. If $u, v \in V(C)$ are
distinct, let $d_{C,F} (u, v)$ denote the smaller of $|E(P )\cap F |$, $|E(Q) \cap F|$ where $P, Q$ are the two paths of
$C$ between $u, v$. A loaded cycle $(C, F )$ is geodesic
in $G$ if $d_G(u, v) \ge d_{C,F} (u, v)$ for all $u, v \in V(C)$. If $G$ admits a tree-decomposition in which all bags
have bounded diameter, then every geodesic loaded cycle has bounded load. The main theorem of \cite{BerSey2024} 
says that if every geodesic loaded cycle has bounded load, then $G$ admits a tree-decomposition in
which all bags have bounded diameter. Let  $\glc(G)$ be the maximum load over all geodesic loaded cycles in $G$. 

\begin{proposition} [\cite{BerSey2024}] \label{prop:sey-lgc} For every graph $G$, $\tl(G)- 1 \le \glc(G) \le 3\cdot\tl(G)$.
\end{proposition}

The notion of bounded load geodesic cycles and Proposition   \ref{prop:sey-lgc} are central in the proofs of results of \cite{BerSey2024} mentioned in Section \ref{sec:embed-to-tree}, Section \ref{sec:early-bn} and Section \ref{sec:early-mcw}. Using a layering partition and its cluster-diameter, we prove those results in a more intuitive way.  Furthermore, in Section \ref{sec:cbc}, we introduce a new natural "bridging`` property for cycles which generalizes a known property of cycles in chordal graphs and show that it also coarsely describes the tree-length. 

%- through geodesic loaded cycles; {\color{red}a less intuitive parameter} \\

%\subsection{Complexities}

%\section{Graphs with bounded tree-length: proofs} 
\section{New proofs for coarse equivalency with tree-length} %of those parameters} 
\label{sec:proofs}
\subsection{Bottleneck constant of a graph}\label{sec:bnc}

Let us define a bottleneck property in its fullness. The {\em overall bottleneck constant}, denoted by $\BNC(G)$, of a graph $G$ is the least integer $r$ such that if $P(u,v)$ is a shortest path of $G$ between $u, v$ and $w$ is a vertex of $P(u,v)$, then every path between $u, v$ contains a vertex that has distance at most $r$ from $w$. It is easy to show that for every graph $G$, $\bn(G)=\BNC(G)$ (see Corollary  \ref{cor:ineq} below). We prove first the following lemma. 

\begin{lemma} \label{lm:BNC_ClusterDiam}
	For every graph $G$,  $\BNC(G)\leq\frac{\widehat{\Delta}(G)}{2}\leq\frac{3 }{2}\tl(G)$.
\end{lemma}

\begin{proof}
	Let $x,y$ be arbitrary vertices of $G$. Consider an arbitrary shortest path $P(x,y)$ connecting $x$ and $y$ and an arbitrary vertex $c$ of $P(x,y)$. 
 %Let $\mathcal{LP}(G,c)$ be a layering partition of $G$ starting at vertex $c$. 
 Let also $Q$ be an arbitrary path between $x$ and $y$ and $\ell \geq 0$ be the maximum integer such that $D_{\ell}(c)$ does not intersect $Q$. 
 
 Consider two vertices $a$ and $b$ in $P(x,y)$ such that $d_G(c,a)=d_G(c,b)=\ell$, $a$ is between $c$ and $x$ and $b$ is between $c$ and $y$. Let also $a'$ and $b'$ be vertices in $P(x,y)$ such that $d_G(c,a')=d_G(c,b')=\ell+1$, $a$ is adjacent to $a'$ and $b$ is adjacent to $b'$ (see Fig. \ref{fig:one}(a) for an illustration). 
 Since $D_{\ell+1}(c)$ intersects $Q$ but $D_{\ell}(c)$ does not, vertices $a',b'$ are in the same cluster of layering partition $\mathcal{LP}(G,c)$ of $G$ starting at vertex $c$ (they are connected outside the disk $D_\ell(c)$ by $Q$ and parts of $P(x,y)$ between $x$ and $a'$ and between $y$ and $b'$). Hence, $d_G(a',b')\le \Delta_c(G)\le \widehat{\Delta}(G)$. On the other hand,
	$d_G(a',b')=1+d_G(a,b)+1=2\ell+2$, i.e., $d_G(c,Q)=\ell+1= \frac{d_G(a',b')}{2}\le \frac{\widehat{\Delta}(G)}{2}$. 
 
 Thus, $\BNC(G)\leq\frac{\widehat{\Delta}(G)}{2}$ must hold. Applying also Proposition \ref{prop:dorisb}, we get $\BNC(G)\leq\frac{\widehat{\Delta}(G)}{2}\leq\frac{3}{2}\tl(G)$. \qed
\end{proof}

Combining Proposition \ref{prop:ClustDiamAtAnys} with Lemma \ref{lm:BNC_ClusterDiam}, one gets the following corollary.

\begin{corollary} \label{cor:ineq}
	For every graph $G$ and every vertex $s$ of $G$,   $\bn(G)= \BNC(G)\leq \frac{3}{2}\Delta_s(G)$. 	
 In particular, $\bn(G)= \BNC(G)\leq \frac{3}{2}\Delta(G)$ for every graph $G$.
\end{corollary} 
\begin{proof} By Proposition \ref{prop:ClustDiamAtAnys} and Lemma \ref{lm:BNC_ClusterDiam}, it remains only to show $\bn(G)= \BNC(G)$. Clearly, by definitions, $\bn(G)\le\BNC(G)$. To show the equality, %at $\BNC(G)\le\bn(G)$ holds, 
let $\BNC(G)=r$. Then, for some vertices $x,y$ in $G$, some shortest path $P(x,y)$ connecting $x$ and $y$, and some vertex $c\in P(x,y)$,  there must exist a path $Q$ in $G[V\setminus D_{r-1}(c)]$ connecting $x$ and $y$. Necessarily, $d_G(x,c), d_G(y,c)\ge r$. Consider a vertex $x'$  on a subpath of $P(x,y)$ between $x$ and $c$ and a vertex $y'$  on a subpath of $P(x,y)$ between $y$ and $c$ with $d_G(x',c)=d_G(y',c)=r$. We have that vertices $x'$ and $y'$ with $d_G(x',y')=2r$ and $d_G(x',c)=d_G(c,y')=r$ are connected in $G[V\setminus D_{r-1}(c)]$ by  $P(x,x')\cup Q\cup P(y,y')$, where $P(x,x')$ and $P(y,y')$ are subpaths of $P(x,y)$ between corresponding vertices.  This shows that $\bn(G)>r-1$, implying $\bn(G)=r=\BNC(G)$. 
\qed  
\end{proof}

  \begin{figure}[htb]%[tbh] %
    \begin{center} %\vspace*{-1mm}
      \begin{minipage}[b]{16cm}%5
        \begin{center} %\hspace*{10mm}
          \vspace*{-36mm}
          \includegraphics[height=16cm]{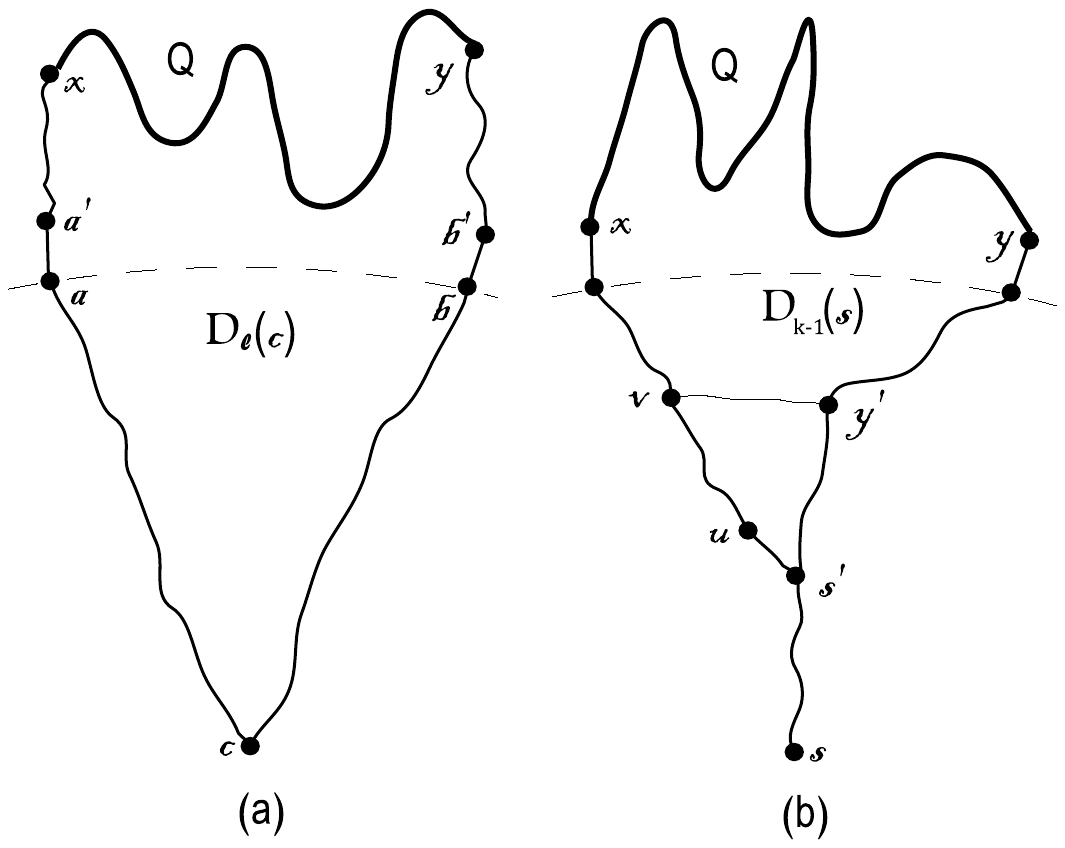}
        \end{center} \vspace*{-49mm}
        \caption{\label{fig:one} Illustrations to the proofs of Lemma \ref{lm:BNC_ClusterDiam} and Lemma \ref{lm:ClusterDiam_bnc}.} %
      \end{minipage}
    \end{center}
   \vspace*{-5mm}
  \end{figure}
% \medskip

Now we upperbound $\Delta_s(G)$ by a linear function of $\bn(G)$.

\begin{lemma} \label{lm:ClusterDiam_bnc}
	For every graph $G$  and every vertex $s$ of $G$, $\Delta_s(G)\leq 4\cdot \bn(G)+2$. In particular, $\widehat{\Delta}(G)\leq 4\
 \bn(G)+2$ for every graph $G$.
\end{lemma}

\begin{proof}
Let $s$ be an arbitrary vertex of $G$ and $\mathcal{LP}(G,s)$ be the layering partition of $G$ starting at $s$. Consider vertices $x$ and $y$ from a cluster of $\mathcal{LP}(G,s)$ with $d_G(x,y)=\Delta_s(G)$, and let $k:=d_G(s,x)=d_G(s,y)$ and $r:=\bn(G)$. Choose also a path $Q$ connecting $x$ and $y$ outside the disk $D_{k-1}(s)$.

Consider arbitrary  shortest paths $P(s,x)$ and $P(s,y)$  of $G$ connecting $s$ with $x$ and $y$, respectively. Let $s'$ be a vertex from  $P(s,x)\cap P(s,y)$ furthest from $s$. We may assume that $k':=d_G(x,s')=d_G(y,s')=k-d_G(s,s')$ is greater than $2r+1$ since, otherwise,  $d_G(x,y)\le d_G(x,s')+d_G(s',y)=2k'\leq 4r+2$, and we are done. 

Pick now a vertex $u$ in $P(s,x)$ such that  
$d_G(x,u)=2r+2\leq d_G(x,s')$. See Fig. \ref{fig:one}(b) for an illustration.  
%and a shortest path $P(s',y)$ connecting $s'$ with $y$.  
The middle vertex $v$ of a subpath $P(u,x)$ of $P(s,x)$ must see within $r$ every path connecting $u$ with $x$. In particular, $d_G(v,Q')\leq r$ for a path $Q'$ formed  by concatenating $Q$ with $P(y,s')\cup P(s',u)$, where $P(y,s')$ is a subpath of $P(y,s)$ between $y$ and $s'$ and $P(s',u)$ is a subpath of $P(s,x)$ between $s'$ and 
$u$. As $d_G(v,Q)=d_G(v,x)=r+1$ and $d_G(v,P(s',u))=d_G(v,u)=r+1$, $v$ cannot see within $r$ any vertex of $Q\cup P(s',u)$. Hence, there must exist a 
vertex $y'$ in $P(y,s')$ such that $d_G(v,y')\le r$. 

Since $d_G(y,y')=k-d_G(s,y')$ and $k=d_G(x,s)\le d_G(s,y')+d_G(y',v)+d_G(v,x)$, necessarily, $d_G(y,y')\le d(y',v)+d_G(v,x)\le r+r+1=2r+1$. Consequently,  $d_G(x,y)\le d_G(x,v)+d_G(v,y')+d_G(y',y)\leq r+1+r+2r+1=4r+2$. That is,  $\Delta_s(G)\leq 4\cdot \bn(G)+2$. \qed
\end{proof}

Combining Lemma  \ref{lm:BNC_ClusterDiam}, Lemma  \ref{lm:ClusterDiam_bnc}, Corollary \ref{cor:ineq} and Proposition \ref{prop:dorisb}, we get the following result. 

\begin{theorem} \label{th:bnc-delta-tl}
	For every graph $G$  and every vertex $s$ of $G$, the following inequalities hold: 
  $$\frac{\Delta_s(G)-2}{4}\leq \bn(G)= \BNC(G)\leq \frac{3}{2}\Delta_s(G),$$ 
   $$\frac{\tl(G)-3}{4}\leq\frac{\widehat{\Delta}(G)-2}{4}\leq \bn(G)=\BNC(G)\leq\frac{\widehat{\Delta}(G)}{2}\leq\frac{3 }{2}\tl(G).$$
 % $\Delta_s(G)/3 \leq \tl(G) \leq \Delta_s(G)+1.$\\
\end{theorem} 

\begin{corollary} \label{cor:ineq-tl-bnc}
	For every graph $G$,   $\frac{2}{3}\bn(G)\leq \tl(G)\leq 4\cdot \bn(G)+3$. 	
\end{corollary}

Same inequalities as in Corollary \ref{cor:ineq-tl-bnc} can be derived from  \cite[Lemma 2.4, Lemma 4.4, Theorem 4.5]{BerSey2024}. 
%{\color{red}/* same in my and  in {BerSey2024} */}

Similar to parameters $\Delta_s(G)$ and $\rho_s(G)$, the bottleneck constant $\bn(G)$ of a given graph $G$ can be computed in  polynomial time (in at most $O(n^3m)$ time, see Section \ref{sec:mcw}). However, it gives a worse than 3-approximation of $\tl(G)$, $\tb(G)$. 

%{\color{red}  Computation of $\bn(G)$? $\BDS(G)$ can be computed in $O(nm)$ time \cite{tree-spanner-appr}. }

\subsection{McCarty-width of a graph}\label{sec:mcw}
Here, we give an alternative proof for Rose McCarty's conjecture that $\tl(G)$ is small if and only if $\mc(G)$ is small. See also Corollary   \ref{cor:mcw-adt} and Theorem \ref{th:mf-tl-mcw} for %more 
other alternative proofs. 

\begin{lemma} \label{lm:ClusterDiam_mcw}
	For every graph $G$  and every vertex $s$ of $G$, $\Delta_s(G)\leq 6\cdot \mc(G)$. In particular, $\widehat{\Delta}(G)\leq 6\cdot \mc(G)$ for every graph $G$.
\end{lemma}
\begin{proof}
Let $s$ be an arbitrary vertex of $G$ and $\mathcal{LP}(G,s)$ be the layering partition of $G$ starting at $s$. Consider vertices $x$ and $y$ from a cluster of $\mathcal{LP}(G,s)$ with $d_G(x,y)=\Delta_s(G)$, and let $k:=d_G(s,x)=d_G(s,y)$ and $r:=\mc(G)$. Choose also a path $Q$ connecting $x$ and $y$ outside the disk $D_{k-1}(s)$ and arbitrary shortest paths $P(s,x)$ and $P(s,y)$ connecting $s$ with $x$ and $y$, respectively. 

Since $r:=\mc(G)$, for vertices $s,x,y$, there is a vertex $v$ in $G$ such that disk $D_r(v)$ intersects all three paths: $Q$, $P(s,x)$ and $P(s,y)$ (see Fig. \ref{fig:two}). Choose arbitrary vertices $u\in Q\cap D_r(v)$, $x'\in P(s,x)\cap D_r(v)$ and $y'\in P(s,y)\cap D_r(v)$. We have $d_G(x',u), d_G(y',u)$ and $d_G(x',y')$  are all at most $2r$.  Since $d_G(s,Q)= k$, we get $d_G(s,x)=k\le d_G(s,u)\le  d_G(s,x')+d_G(x',u)\le d_G(s,x')+2r$. That is, $d_G(x,x')=d_G(s,x)-d_G(s,x')\le 2r$. Similarly, $d_G(y,y')\le 2r$. Consequently, $d_G(x,y)\le d_G(x,x')+d_G(x',y')+d_G(y,y')\le 6r$, i.e., $\Delta_s(G)\le 6\cdot \mc(G)$.   
\qed
\end{proof} 

  \begin{figure}[htb]%[tbh] %
    \begin{center} %\vspace*{-1mm}
      \begin{minipage}[b]{16cm}%5
        \begin{center} %\hspace*{10mm}
          \vspace*{-36mm}
          \includegraphics[height=16cm]{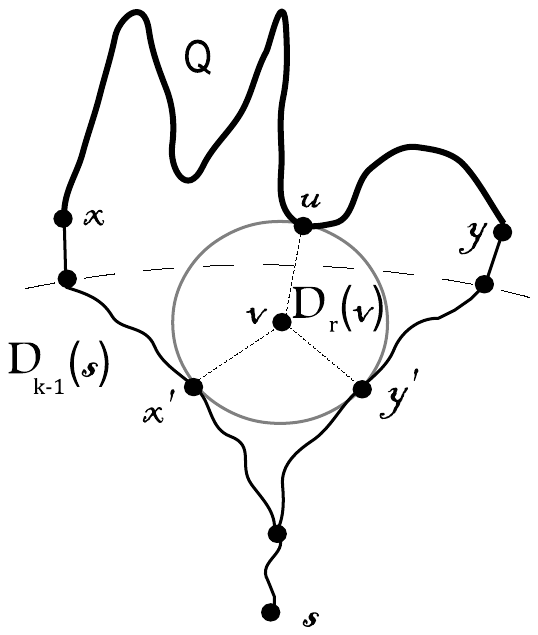}
        \end{center} \vspace*{-73mm}
        \caption{\label{fig:two} Illustrations to the proof of Lemma \ref{lm:ClusterDiam_mcw}. }  
      \end{minipage}
    \end{center}
   \vspace*{-5mm}
  \end{figure}

\begin{lemma} \label{lm:Clusterrad_BDS}
	For every graph $G$, every vertex $s$ of $G$, and every integer $k\ge 3$, $\mc_k(G)\le \rho_s(G)$.  In particular, $\mc_k(G)\leq \rho(G)$ for every graph $G$ and every integer $k\ge 3$. Furthermore, for any subset $X\subseteq V$ of vertices of $G$, a balanced disk separator $D_r(u)$ with $r\le \Delta_s(G)$ can be found in linear time. 
\end{lemma}

\begin{proof}
%We need to show only the right inequality. 
Let $s$ be an arbitrary vertex of $G$ and $\mathcal{LP}(G,s)=\{L^i_1,\ldots,L^i_{p_i}:i=0,1,\dots,q\}$  be the layering partition of $G$ starting at $s$. Consider also the  layering tree $\Gamma:=\Gamma(G,s)$ of graph $G$ with respect to the layering partition $\mathcal{LP}(G,s)$.  

For a given subset $X\subseteq V$ of vertices of $G$, we can assign to each node $L_i^j$ of $\Gamma$  a weight $w_i^j: =|L_i^j\cap X|$. Clearly, $W:=\sum_{i=0,1,2,\ldots,q, j=1,2,\ldots,p_i}w_i^j$ is equal to  $|X|$.
It is known that every node-weighted tree $T$ with the total weight of nodes equal to $W$ has a node $x$, called a {\em median} of $T$,
such that the total weight of nodes in each subtree of $T\setminus \{x\}$ does not exceed $W/2$. Furthermore, such a node $x$ of $T$ can be found in $O(|V(T)|)$ time \cite{goldman}. 

Let $C=L_i^j$ $(i\in \{0,1,2,\ldots,q\}, j\in \{1,2,\ldots,p_i\})$ be a median node of weighted tree $\Gamma$.   Then, each subtree of $\Gamma\setminus \{C\}$ has total weight of nodes not exceeding $|X|/2$. It is clear from the construction of tree $\Gamma$ that the set $C\subseteq V$ separates in $G$ any two vertices that belong to clusters from different subtrees of $\Gamma\setminus \{C\}$. Consequently, $C$ is a balanced separator of $G$ with respect to $X$ as any connected component of $G[V\setminus C]$ has no more than $|X|/2$ vertices from $X$. Note that, given a graph $G$, such a cluster $C$ of layering partition ${\mathcal LP}(G,s)$ of $G$
can be found in linear time in the size of $G$. 

Since there is a vertex $v$ in $G$ such that $C\subseteq D_r(v)$ for $r\le \rho_s(G)$, clearly, $D_r(v)$ is a balanced disk separator of $G$ with respect to $X$. As this holds for an arbitrary subset $X\subseteq V$, we conclude $\mc_k(G)\le \rho_s(G)$ for any $k\ge 3$. 

It is not clear how to find in $o(nm)$ time a vertex $v$ such that $C\subseteq D_r(v)$ (we may need the entire distance matrix of $G$ to do that). Instead, we can take an arbitrary vertex $u$ from $C$ and run a breadth-first-search $BFS(u)$ from $u$ in $G$ until we reach the layer $L^i$ of $BFS(u)$ with minimum $i$ such that $C\subseteq \bigcup_{j=0}^iL^i$. Clearly, $i\leq \Delta_s(G)$ and $D_i(u)$ is a balanced disk separator of $G$ with respect to $X$. 
\qed
\end{proof}
In \cite{tree-spanner-appr}, such a balanced disk separator is used to obtain an $O(m\log n)$-time $O(\log n)$-approximation algorithm for the problem of finding a tree $t$-spanner with minimum $t$ of a given graph $G$.

It is easy to show that  in fact $\mc(G)\leq \tb(G)$ holds (see \cite{BerSey2024}). To show a stronger inequality $\mc_k(G)\le \tb(G)$ for every $k\ge 3$,  we will need %some notions from the Hypergraph Theory and 
a nice result from \cite{bal-clique-ch} on balanced clique-separators of chordal graphs.  See also Proposition   \ref{prop:bramble+} for an alternative proof. 
%(graphs that do not contain any induced cycles on four or more vertices).  

\begin{lemma} [\cite{bal-clique-ch}] \label{lm:clique_sep}
	For every chordal graph $G=(V,E)$ and every subset $X\subseteq V$ of vertices of $G$, there is a clique $C$ in $G$ such that if the vertices of $C$  are removed from $G$, every connected component in the graph induced by any remaining vertices has at most $|X|/2$ vertices from $X$.  
\end{lemma}

Using Lemma  \ref{lm:clique_sep}, we can prove the following {\em balanced-disk separator lemma}. 

\begin{lemma} \label{lm:tb_mcw_k} 
	For every graph $G=(V,E)$ and every subset $X\subseteq V$ of vertices of $G$, there is a vertex $v\in V$ such that if the vertices of disk $D_{\beta}(v)$, where $\beta=\tb(G)$, are removed from $G$, every connected component in the graph induced by any remaining vertices has at most $|X|/2$ vertices from $X$.  Consequently,  $\mc_k(G)\le \tb(G)$ holds for every graph $G$  and every integer $k\ge 3$.  
\end{lemma}
\begin{proof} Such a result for the case when $X=V$ was first proved in \cite{tree-spanner-appr}. We will adapt that proof to an arbitrary $X\subseteq V$.  
Let $G$ be a graph with $\tb(G)=\beta$ and $\cT(G)$ be its tree-decomposition  of breadth $\beta$. We can construct a new graph $G^+$ from $G$ by adding an edge between every two distinct non-adjacent vertices $x,y$ of $G$ such that a bag $B$ in $\cT(G)$ exists with $x,y\in B$.  Clearly, $G^+$ is a {\em supergraph} of the graph $G$ (each edge of $G$ is an edge of $G^+$, but $G^+$ may have some extra edges between non-adjacent vertices of $G$ contained in a common bag of $\cT(G)$). 
It is known (see, e.g., survey \cite{Bodlaender}) that $G^+$ is a chordal graph, $\cT(G)$ is a clique-tree of $G^+$,  and for each clique $C$ of $G^+$ there is a bag $B$ in $\cT(G)$ such that $C\subseteq B$. 
By Lemma  \ref{lm:clique_sep},  the chordal graph  $G^{+}$ contains a balanced clique-separator $C$.
Since $C$ is contained in a bag of $\cT(G)$, there must exist a vertex $v\in V(G)$ with $C\subseteq D_{\beta}(v,G)$. As the removal of  the vertices of $C$ from $G^+$ leaves no connected component in $G^+[V\setminus C]$ with more that $|X|/2$ vertices from $X$, and since $G^+$ is a supergraph of $G$, clearly, the removal of  the vertices of $D_{\beta}(v,G)$ from $G$ leaves no connected component in $G[V\setminus D_{\beta}(v,G)]$ with more that $|X|/2$ vertices from $X$.
\qed
\end{proof}

Note that, to find a balanced disk-separator $D_r(v)$ with $r\le \tb(G)$ of a graph $G$ with respect to a subset $X\subseteq V$, one does not need to have a tree-decomposition $\cT(G)$ of breadth $\tb(G)$. For a given graph $G=(V,E)$, a subset $X\subseteq V$ and a fixed vertex $v$ of $G$,  a balanced disk-separator $D_r(v)$ with minimum $r$ can be computed in $O(m)$ time (see \cite{tree-spanner-appr}). Hence, an overall  balanced disk-separator $D_r(v)$ with minimum $r$ of a graph $G$ with respect to a subset $X\subseteq V$ can be found in total $O(nm)$ time (one can run the algorithm from  \cite{tree-spanner-appr} for every $v\in V$ and pick a best vertex $v$).  In \cite{tree-spanner-appr}, such a balanced disk separator is used to obtain an $O(nm\log^2 n)$-time $(\log_2 n)$-approximation algorithm for the problem of finding a tree $t$-spanner with minimum $t$ of a given graph $G$.

Combining Proposition \ref{prop:dorisb}, Lemma  \ref{lm:ClusterDiam_mcw}, Lemma  \ref{lm:Clusterrad_BDS}, and Lemma  \ref{lm:tb_mcw_k},  %??? Corollary \ref{cor:ineq} and Proposition \ref{prop:dorisb}, 
we get the following result. 
\begin{theorem} \label{th:mcw-delta-rho}
	For every graph $G$  and every vertex $s$ of $G$, the following inequalities hold: 
$$\mc(G)\leq \rho(G)\le\rho_s(G)\le \Delta_s(G)\le\widehat{\Delta}(G)\leq 6\cdot \mc(G),$$ 
$$\mc(G)\leq \tb(G)\leq \tl(G)\leq  \Delta_s(G)+1\leq 6\cdot \mc(G)+1.$$ 	
\end{theorem} 

In \cite[Theorem 5.1]{BerSey2024}, it was shown that $\mc(G)\leq \tb(G)\leq \tl(G)\leq 6\cdot \mc(G)+3$ holds. 

%{\color{red}/* mine is slightly better than in {BerSey2024} */}

From Theorem \ref{th:bnc-delta-tl} and Theorem  \ref{th:mcw-delta-rho}, we also get. 
\begin{corollary}  \label{cor:ineq-bnc-mcw}
For every graph $G$,   $\mc(G)\leq  4\cdot \bn(G)+2$ and $\bn(G)\le  3\cdot \mc(G)$. 
\end{corollary}

It is easy to see that, for every graph $G$ and every $k\ge 2$,  $\mc_{2k}(G)\le\mc_{2k-1}(G)$. Indeed, consider a set $X=\{x_1,\dots,x_{2k-1},$ $x_{2k}\}\subseteq V$ and let $r:=\mc_{2k-1}(G)$. Then, there is a vertex $v\in V$ such that every connected component of $G[V\setminus D_r(v)]$ has at most $\frac{2k-1}{2}$, i.e., at most $k-1$, vertices from $X\setminus \{x_{2k}\}$. Hence, those connected components contain at most $k$ vertices from $X$. So, $\mc_{2k}(G)\le\mc_{2k-1}(G)$ must hold. 
Two induced cycles $C_3$ (with 3 vertices each) sharing one common vertex (a so-called {\em bow-graph}) gives an example of a graph $G$ with $\mc_3(G)=1$ and $\mc_4(G)=0$. So,  $\mc_{2k}(G)<\mc_{2k-1}(G)$ is possible for some graph $G$.
We do not know the relationship between $\mc_{2k}(G)$ and $\mc_{2k+1}(G)$, except that there are graphs $G$ and integers $k$ such that $\mc_{2k}(G)<\mc_{2k+1}(G)$ (e.g., for a block graph with three blocks and two articulation points, each block being a triangle, $\mc_{4}(G)=0<1=\mc_{5}(G)$ holds). However, from our results, it follows that, for every graph $G$, every $k\ge 3$ and $s\in V$, $\mc_k(G)\le\rho_s(G)\le\Delta_s(G)\le 6\cdot\mc_3(G)=6\cdot\mc(G)$ (which is of independent interest). 
Thus, for every $k\ge 3$, $\mc_k(G)$ is bounded from above by a (linear) function of $\mc(G)$. This leads to a natural question if $\mc_k(G)$ ($k> 3$) and $\mc_3(G)$ are coarsely equivalent parameters. Unfortunately, the answer is 'no'. By extending our bow-graph example, we can show that $\mc_k(G)$ ($k>3$), generally, cannot be bounded from below by a function of $\mc(G)$.  
%Furthermore, ...  the difference between $\mc_k(G), k>3,$ and $\mc(G)$ could be arbitrarily large. 
Consider a graph $G$ consisting of two induced cycles $C_{6p}$ on $6p$ $(p\ge 1)$ vertices each,  sharing just one common vertex $v$. We have $n:=|V(G)|=12p-1$ and  $\mc_n(G)=0$ (we just need to remove vertex $v$ to partition $G$ into two connected components each having $6p-1$ vertices from $V(G)$). Considering also three vertices of one $C_{6p}$ that are pairwise at distance $2p$ from each other, we get $\mc(G)\ge p$.  %. It is easy to check that

As we have mentioned earlier, a balanced disk-separator $D_r(v)$ with minimum $r$ of a graph $G$ with respect to any subset $X\subseteq V$ can be found in total $O(nm)$ time. Consequently, for every graph $G$, the parameter $\mc(G)$ can be computed in at most $n^3\cdot O(nm)=O(n^4m)$ time. Similarly, for every graph $G$, the parameter $\bn(G)$ can be computed in at most $O(n^3m)$ time. One can try to get better time complexities (using some special data structures), but since $\bn(G)$ and $\mc(G)$ are further from $\tl(G)$ than $\Delta_s(G)$ is, we did not pursue this line of investigation. 

%{\color{red}
%\begin{question} Should $\mc_k(G)$ be incorporated in the inequalities? What is the gap between $\mc_k(G)$ and $\mc_{k+1}$? Should $\mc_k(G)$ be redefined for all $|X|\le k$? Computation of $\mc(G)$? In a straightforward way, in $O(n^3\times nm)=O(n^4m)$ time.
%\end{question}
%}

\subsection{Brambles and Helly families of connected subgraphs or paths}\label{sec:br-Helly} 
A {\em bramble} of a graph $G$ is a family of connected subgraphs of $G$ that all {\em touch} each other: for every pair of disjoint subgraphs, there must exist an edge in $G$ that has one endpoint in each subgraph. The {\em order} of a bramble is the smallest size of a hitting set, a set of vertices of $G$ that has a nonempty intersection with each of the subgraphs. Brambles are used to characterize the tree-width of $G$: $k$ is the largest possible order among all brambles of $G$ if and only if $G$ has tree-width $k - 1$ \cite{SeyThom1993}.  

A family of %connected 
subgraphs of $G$ 
is called a {\em Helly family} if every two subgraphs of the family intersect. Clearly, any Helly family of connected subgraphs of $G$ is a particular bramble of $G$. 
Below we define new properties for brambles and Helly families of connected subgraphs of $G$ that coarsely define the tree-breadth $\tb(G)$. 

Let ${\cal F}:=\{H_1,\dots,H_p\}$ be a family of subgraphs of $G$. We say that a disk $D_r(v)$ of $G$ intercepts all members of ${\cal F}$ if $V(H_i)\cap D_r(v)\neq \emptyset$ for every $i=1,\dots,p$.  
Denote by $\br(G)$ (the {\em bramble interception} radius  of $G$) the smallest integer $r$ such that for every bramble ${\cal F}$ of $G$, there is a disk $D_r(v)$ which intercepts all members of ${\cal F}$.  Denote by $\sh(G)$ (by $\ph(G)$) the smallest integer $r$ such that for every Helly family ${\cal F}$ of connected subgraphs (of paths, respectively) of $G$, there is a disk $D_r(v)$ which intercepts all members of ${\cal F}$. Call $\sh(G)$ ($\ph(G)$) the interception radius for Helly families of connected subgraphs (of  paths) of $G$. Such property for Helly families of disks of a graph were already considered in literature (see, e.g., \cite{HellyGroups,ChepoiEst,gap}). 

It turns out that all these three parameters are coarsely equivalent to tree-breadth. We prove this by involving the McCarty-width parameter and a result from \cite{SeyThom1993}. 

\begin{proposition}   \label{prop:bramble}
	For every $G$,  $\mc(G)\le \ph(G)\le\sh(G)\le\br(G)\le \tb(G)\le 6\cdot \mc(G)+1$. 
\end{proposition} 
\begin{proof} Inequalities $\ph(G)\le\sh(G)\le\br(G)$ follow from definitions. Hence, by Theorem \ref{th:mcw-delta-rho}, we only need to show $\br(G)\le \tb(G)$ and $\mc(G)\le \ph(G)$. According to %\cite[(2.3)]{SeyThom1993}, 
\cite{SeyThom1993},  for every bramble ${\cal F}$ of a graph $G$, there is a bag $B$ in every tree-decomposition of $G$ which intercepts all members of ${\cal F}$. Hence, for a tree-decomposition with minimum breadth, there is a bag $B$ and a vertex $v$ in $G$ with $B\subseteq D_r(v)$, $r\le \tb(G)$, such that the disk $D_r(v)$  of $G$  intercepts all members of ${\cal F}$. Consequently, $\br(G) \le \tb(G)$. 

Consider now any three vertices $x,y,z$ of $G$ and the family ${\cal F}$ of all paths of $G$ connecting pairs from $\{x,y,z\}$. The family ${\cal F}$ is a Helly family. If $\ph(G)=r$, then there is a disk $D_r(v)$ in $G$ that  intercepts all paths from ${\cal F}$. Consequently, no connected component of $G[V\setminus D_r(v)]$  contains two of $x,y,z$. The latter proves $\mc(G)\le \ph(G)$. 
\qed
\end{proof}

We can generalize the second part of the proof of Proposition \ref{prop:bramble} and show $\mc_k(G)\le \sh(G)$ for every graph $G$ and every $k\ge 3$. This provides also an alternative proof of Lemma  \ref{lm:tb_mcw_k}. 
\begin{proposition}   \label{prop:bramble+}
	For every graph $G$ and every $k\ge 3$,  $\mc_k(G)\le \sh(G)\le\br(G)\le \tb(G)$. 
\end{proposition} 
\begin{proof} We need only to show $\mc_k(G)\le \sh(G)$. Consider a subset $X=\{x_1,x_2,\dots,x_k\}$ of vertices of $G$ and any subset $Y\subset X$ containing $\lfloor\frac{k}{2}\rfloor+1$ vertices of $X$. Let 
${\cal F}_{Y}$ be the family of all connected subgraphs of $G$ spanning vertices of $Y$. For any two subsets $X'$ and $X''$ of $X$ containing  $\lfloor\frac{k}{2}\rfloor+1$ vertices of $X$ each, any subgraph from ${\cal F}_{X'}$ intersects any subgraph from ${\cal F}_{X''}$ (as $X'\cap X''\neq\emptyset$). Hence, ${\cal F}:= \cup \{{\cal F}_{Y}: Y\subset X, |Y|=\lfloor\frac{k}{2}\rfloor+1\}$ is a Helly family  of connected subgraphs of $G$.  
Hence, there must exist a vertex $v$ in $G$ such that disk $D_r(v)$, with $r\le \sh(G)$, intercepts each subgraph from ${\cal F}$.  The latter implies that no connected component of $G[V\setminus D_r(v)]$  contains more than $\lfloor\frac{k}{2}\rfloor$ vertices from $X$. Thus, $\mc_k(G)\le \sh(G)$ must hold. 
\qed
\end{proof}
%%%%%%%%%%%% 
For an induced cycle $C_6$ on six vertices, we have $\ph(C_6)=\sh(C_6)=1$ while $\br(C_6)=2$. So, $\sh(G)<\br(G)$ is possible for some graph $G$. %?????????????????????????????????????????

\subsection{Distance $k$-approximating trees}\label{sec:adt}
By Proposition \ref{lem:cluster-diam} and Proposition \ref{prop:dorisb}, $\adt(G)\leq \Delta(G)\leq 3\cdot \tl(G)$. Hence, we need only to upperbound $\tl(G)$ by a linear function of $\adt(G)$. 

\begin{lemma}   \label{lm:tl-adt}
	For every graph $G$,  $\tl(G)\le 2\cdot \adt(G)+1$. 
\end{lemma} 
\begin{proof}
Let $G=(V,E(G))$, $\adt(G)=r$, and $T=(V,E(T))$ be a tree such that $|d_G(x,y)-d_T(x,y)|\le r$ for every $x,y\in V$. For every edge $uv\in E(G)$, $d_T(u,v)\le d_G(u,v)+r\le 1+r$ holds. Hence, $G$ is a spanning subgraph of graph $T^{r+1}$, where  $T^{r+1}$ is the $(r+1)^{st}$-power of $T$. It is known (see e.g., \cite{Andreas-book,golumbic}) that every power of a tree is a chordal graph. Consequently, there is a clique-tree $\cT(G)$ of $T^{r+1}$.  Clearly, $\cT(G)$ is a tree-decomposition of $G$ such that, for every two vertices $x$ and $y$ belonging to same bag %$C$ 
of $\cT(G)$, $xy$ is an edge of $T^{r+1}$. Necessarily, $d_T(x,y)\le r+1$, by the definition of the $(r+1)^{st}$-power of  $T$. Furthermore, since $d_G(x,y)\le d_T(x,y)+r\le 2r+1$ for every $x,y$ belonging to same bag of $\cT(G)$, the length of the tree-decomposition $\cT(G)$ of $G$ is at most $2r+1=2\cdot \adt(G)+1$. 
\qed
\end{proof}
Combining Proposition \ref{lem:cluster-diam} and Proposition \ref{prop:dorisb} with Theorem \ref{th:bnc-delta-tl}, Lemma  \ref{lm:tl-adt} and Lemma \ref{lm:ClusterDiam_bnc}, we get the following result.  

\begin{theorem} \label{th:adt-tl-delta}
	For every graph $G$  and every vertex $s$ of $G$, the following inequalities hold: 
$$\adt(G)\leq \Delta(G)\le \Delta_s(G)\le\widehat{\Delta}(G)\leq 3\cdot \tl(G)\le 6\cdot \adt(G)+3,$$ 
$$\adt(G)\leq \Delta_s(G)\le 4\cdot \bn(G)+2, $$ 
%\mbox{~and~} 
$$\bn(G)\le 3\cdot \adt(G)+1.$$
%$$\mc(G)\leq \tb(G)\leq \tl(G)\leq 2\cdot \adt(G)+1.$$ 
\end{theorem}

\begin{corollary} %[\cite{BerSey2024}] 
\label{cor:ineq-tl-adt}
	For every graph $G$,   $\frac{\tl(G)-1}{2}\leq \adt(G)\leq 3\cdot \tl(G).$
\end{corollary}

Recall that in \cite[Theorem 4.1, Theorem 4.2, Theorem 4.5]{BerSey2024}, it was shown that  $\frac{\tl(G)-2}{2}\leq \ad(G)\leq 6\cdot \tl(G)$  and $\ad(G)\leq 24\cdot \bn(G)+18 $ hold.

From Theorem  \ref{th:mcw-delta-rho} and Theorem \ref{th:adt-tl-delta}, it also follows  $\adt(G)\le\widehat{\Delta}(G)\leq 6\cdot \mc(G)$ and $\mc(G)\le\widehat{\Delta}(G)\leq 6\cdot \adt(G)+3.$ The latter inequality can further be improved to $\mc(G)\le \frac{3\cdot \adt(G)+1}{2}.$ For this we will need the following interesting lemma.

\begin{lemma}   \label{lm:mcw-adt}
Let $G=(V,E)$ be a graph, $T=(V,E')$ be a distance $k$-approximating tree of $G$ with $k:=\adt(G)$.  For every $x,y\in V$, any path $P_G(x,y)$ of $G$ between $x$ and $y$, and any vertex $c$ from the path $P_T(x,y)$ of $T$ between $x$ and $y$, $d_G(c, P_G(x,y))\le \frac{3\cdot \adt(G)+1}{2}$ holds. 
\end{lemma} 
\begin{proof} Removing $c$ from $T$, we separate $x$ from $y$. Let $T_y$ be the subtree of $T[V\setminus \{c\}]$ containing $y$. Since $x\notin T_y$, we can find an edge $ab$ of $P_G(x,y)$ with $a\in T_y$ and
$b\notin T_y$. Therefore, the path $P_T(a,b)$ must go via $c$. If $d_T(c,a) > (k+1)/2$ and $d_T(c,b) >(k+1)/2$, 
then $d_T(a,b) = d_T(a,c) + d_T(c,b) > k+1$ and, since $d_G(a,b) = 1$, we obtain a contradiction with the 
assumption that $T$ is a distance $k$-approximating tree of $G$ (with  $d_T (a,b)\le d_G(a,c)+k= k+1$). Hence $d_T(c,P_G(x,y)) \le \min\{d_T(c,a),d_T(c,b)\} \le (k+1)/2$. 
Let $z$ be a vertex of $P_G(x,y)$ such that $d_T(c,P_G(x,y))=d_T(c,z)$. We have 
$d_G(c,P_G(x,y))\le d_G(c,z)\le d_T(c,z)+k= d_T (c,P_G(x,y))+k\le (k+1)/2+k$. 
\qed
\end{proof}

\begin{corollary}   \label{cor:mcw-adt}
	For every graph $G$,  $\adt(G)\le 6\cdot \mc(G)$ and $\mc(G)\le \frac{3\cdot \adt(G)+1}{2}$. 
\end{corollary} 
\begin{proof} The first inequality follows from Theorem \ref{th:mcw-delta-rho} and Theorem \ref{th:adt-tl-delta}. The second one follows from Lemma    \ref{lm:mcw-adt}. 
Indeed, let $x,y,z$ be three arbitrary vertices of $G$ and $T$ be a distance $k$-approximating tree of $G$ with $k:=\adt(G)$.  Let $c$ be the unique vertex
of $T$ that is on the intersection of paths $P_T(x,y)$, $P_T(x,z)$ and $P_T(y,z)$. Since $c$ belongs to all three paths, applying Lemma \ref{cor:mcw-adt} three times,
we infer that $d_G(c,P_G(x,y))\le \frac{3\cdot \adt(G)+1}{2}, d_G(c,P_G(x,z))\le \frac{3\cdot \adt(G)+1}{2}$ and $d_G(c,P_G(y,z))\le \frac{3\cdot \adt(G)+1}{2}$ for every path $P_G(x,y)$ between $x$ and $y$, for every path $P_G(x,z)$ between $x$ and $z$, and  for every path $P_G(z,y)$ between $z$ and $y$.  So, no connected component of $G[V\setminus D_r(c)]$, $r=\frac{3\cdot \adt(G)+1}{2}$,  contains two of $x,y,z$.
\qed
\end{proof}

\subsection{$K$-Fat $K_3$-minors}\label{sec:fat}
Here, we give an alternative proof for a result from 
\cite{GeorPapa2023} that a graph $G$ has no $K$-fat $K_3$-minor for some constant $K>0$ if and only if $G$ is $(1,C)$-quasi-isometric to a tree for some constant $C$. %$C\le 14K$. 
Our Corollary \ref{cor: adt-fm} improves the constant $C$. % to $C\le 5K-1$.  

\begin{lemma} \label{lm:ClusterDiam_fat}
	Let $G$ be a graph  and $s$ be a vertex  of $G$. If $\Delta_s(G)\ge 5K$ then $G$ has a $K$-fat $K_3$-minor.  
\end{lemma}
\begin{proof}
Let $\mathcal{LP}(G,s)$ be the layering partition of $G$ starting at $s$. Consider vertices $x$ and $y$ from a cluster of $\mathcal{LP}(G,s)$ with $d_G(x,y)=\Delta_s(G)$, and let $\ell:=d_G(s,x)=d_G(s,y)$. Choose also a path $Q$ connecting $x$ and $y$ outside the disk $D_{\ell-1}(s)$ and arbitrary shortest paths $P(s,x)$ and $P(s,y)$ connecting $s$ with $x$ and $y$, respectively. 
Since $\Delta_s(G)\ge 5K$, we have $d_G(x,y)\ge 5K$. We construct a $K$-fat $K_3$-minor of $G$ in the following way. 

We know that $\ell$ must be greater that $2K$ (otherwise, $d_G(x,y)\le d_G(x,s)+d_G(s,y)\le 2K+2K=4K$, contradicting with $d_G(x,y)\ge 5K$). 
Consider vertices $x'$ and $s_x$ on path $P(x,s)$ at distance $K$ and $2K$ from $x$, respectively, i.e., with $d_G(x,x')=d_G(x',s_x)=K$. Similarly, consider  
vertices $y'$ and $s_y$ on path $P(y,s)$ at distance $K$ and $2K$ from $y$, respectively. As three connected subgraphs of $G$ choose $H_x:=G[D_K(x)], H_y:=G[D_K(y)]$ and $H_s:=G[D_{\ell-2K}(s)]$. As three paths choose a subpath $P(x',s_x)$ of $P(x,s)$ between $x'$ and $s_x$, a subpath $P(y',s_y)$ of $P(y,s)$ between $y'$ and $s_y$ and a subpath $Q(x'',y'')$ of $Q$ between vertices $x'',y''\in Q$, where $x''$ is the vertex of $Q\cap D_K(x)$ which maximizes $d_Q(x,x'')$ and $y''$ is the vertex of $Q\cap D_K(y)$ which maximizes $d_Q(y,y'')$. Clearly, those three connected subgraphs and three paths form a $K_3$-minor in $G$ (note that $d_Q(x'',y'')\ge 3K$ since, otherwise, $d_G(x,y)\le d_G(x,x'')+d_G(x'',y'')+d_G(y'',y)<K+3K+K=5K$, which is impossible).  It remains to show that it is a $K$-fat $K_3$-minor. 

We have  $d_G(V(H_x),V(H_s))=d_G(D_K(x),D_{\ell-2K}(s))= K$ since $d(x,s)=\ell=K+K+(\ell-2K).$ Similarly, $d_G(V(H_y),V(H_s))= K$. Furthermore,  $d_G(V(H_x),V(H_y))=d_G(D_K(x),D_{K}(y))\ge 3K$ since $d(x,y)\ge 5K.$ If $d_G(P(x',s_x),Q(x'',y''))<K$ holds, then $d_G(s,Q(x'',y''))\le d_G(s,x')+d_G(P(x',s_x),Q(x'',y''))<\ell-K+K=\ell$. The latter implies $d_G(s,Q)<\ell$, which is impossible. So, $d_G(P(x',s_x),Q(x'',y''))\ge K$ must hold. Similarly, $d_G(P(y',s_y),Q(x'',y''))\ge K$ must hold. If $d_G(P(x',s_x),P(y',s_y))<K$ holds, then $d_G(x,y)\le d_G(x,s_x)+ d_G(P(x',s_x),P(y',s_y))+d_G(s_y,y)<2K+K+2K=5K$, contradicting with  $d_G(x,y)\ge 5K$. If $d_G(V(H_x),P(y',s_y))=d_G(D_K(x),P(y',s_y))<K$, then $d_G(x,y)\le d_G(x,x')+ d_G(D_K(x),P(y',s_y))+d_G(s_y,y)<K+K+2K=4K$, which is impossible. 
So, $d_G(V(H_x),P(y',s_y))\ge K$ and, by symmetry, $d_G(V(H_y),P(x',s_x))\ge K$. Finally, $d_G(V(H_s),Q(x'',y''))=d_G(D_{\ell-2K}(s),Q(x'',y''))\ge 2K$ since, otherwise, we get $d_G(s,Q)<\ell$, which is impossible. 

Thus, constructed $K_3$-minor of $G$ is $K$-fat. 
\qed
\end{proof} 

From Lemma \ref{lm:ClusterDiam_fat} and Theorem \ref{th:adt-tl-delta}, we have the following corollary. 
%, which also improves the constant $C$ in Proposition  \ref{prop:Papaaoglu}.  

\begin{corollary} \label{cor: adt-fm}
If $G$ has no $K$-fat $K_3$-minor, then  $\Delta_s(G)\le 5K-1$ for every vertex $s$ of $G$. In particular, $\adt(G)\le 5K-1$. 
\end{corollary}

%\newpage 
\begin{lemma} \label{lm:mcw_fat}
	Let $G$ be a graph with $\mc(G)=r$. Then, $G$ has no $K$-fat $K_3$-minor for $K>2r$.  
\end{lemma}
\begin{proof} Assume $G$ has a $K$-fat $K_3$-minor ($K>2r$) formed by three connected subgraphs $H_1$, $H_2$, $H_3$ and three simple paths $P_{1,2}$, $P_{2,3}$, $P_{1,3}$ such that for each $i,j\in \{1,2,3\}$ ($i\neq j$), 
\begin{itemize}
    \item[(1)]      $P_{i,j}$ has one end in $H_i$ and the other end in $H_j$ and  $|P_{i,j}\cap V(H_i)|=|P_{i,j}\cap V(H_j)|=1$, and
    \item[(2)]      $d_G(V(H_i),V(H_j))\ge K$,  $d_G(P_{i,j},V(H_k))\ge K$ ($k\in \{1,2,3\}, k\neq i, j$) and the distance between any two paths $P_{1,2}$, $P_{2,3}$, $P_{1,3}$ is at least $K$.  
\end{itemize}   
Let \begin{itemize}
\item $P_{1,2}\cap V(H_1)=\{x_1\}$, $P_{1,2}\cap V(H_2)=\{x_2\}$, 
\item $P_{2,3}\cap V(H_2)=\{y_2\}$, $P_{2,3}\cap V(H_3)=\{y_3\}$,  
\item $P_{1,3}\cap V(H_1)=\{z_1\}$, $P_{1,3}\cap V(H_3)=\{z_3\}$.   
\end{itemize}  
Vertices $x_1$ and $z_1$ are connected in $H_1$ via a path of length at least $K$. Choose such a path $Q_1(x_1,z_1)$ and let $v_1$ be a middle vertex of $Q_1(x_1,z_1)$. 
Similarly, choose a path $Q_2(x_2,y_2)$ in $H_2$ and a middle vertex $v_2$ of $Q_2(x_2,y_2)$, and choose a path $Q_3(y_3,z_3)$ in $H_3$ and a middle vertex $v_3$ of  $Q_3(y_3,z_3)$. 

Since $\mc(G)=r$, for vertices $v_1,v_2,v_3$ there must exist a vertex $u$ in $G$ such that no connected component of $G[V\setminus D_r(u)]$  contains two of $v_1,v_2,v_3$. We show that this is not possible due to $K>2r$ and distance requirements listed in (2). 

Disk $D_r(u)$ needs to intercept each of the following three paths: 
\begin{itemize}
\item $P(v_1,v_2):=Q_1(v_1,x_1)\cup P_{1,2}\cup Q_2(x_2,v_2)$, where $Q_1(v_1,x_1)$ and $Q_2(x_2,v_2)$ are subpaths of $Q_1(z_1,x_1)$ and $Q_2(x_2,y_2)$, respectively, connecting corresponding vertices, 
\item $P(v_2,v_3):=Q_2(v_2,y_2)\cup P_{2,3}\cup Q_3(y_3,v_3)$, where $Q_2(v_2,y_2)$ and $Q_3(y_3,v_3)$ are subpaths of $Q_2(x_2,y_2)$ and $Q_3(y_3,z_3)$, respectively, connecting corresponding vertices, 
\item $P(v_1,v_3):=Q_1(v_1,z_1)\cup P_{1,3}\cup Q_3(z_3,v_3)$, where $Q_1(v_1,z_1)$ and $Q_3(z_3,v_3)$ are subpaths of $Q_1(x_1,z_1)$ and $Q_3(z_3,y_3)$, respectively, connecting corresponding vertices. 
\end{itemize}  
Choose $w_{1,2}\in D_r(u)\cap P(v_1,v_2)$, $w_{1,3}\in D_r(u)\cap P(v_1,v_3)$, $w_{2,3}\in D_r(u)\cap P(v_2,v_3)$. 
Vertices $w_{1,2},w_{1,3},w_{2,3}$ are pairwise at distance at most $2r<K$ in $G$.  If $w_{1,2}\in P_{1,2}$ then, by distance requirements in (2),  $w_{2,3}$ can neither be in  $P_{2,3}$  nor in $H_3$. Hence, $w_{2,3}$ is in $Q_2(v_2,y_2)\subset V(H_2)$. Similarly, $w_{1,3}$ must be in $Q_1(v_1,z_1)\subset V(H_1)$. But then, we get $d_G(V(H_1),V(H_2))\le 2r<K$, which is not possible. So, by symmetry, we can assume $w_{1,2}\notin P_{1,2}$, $w_{1,3}\notin P_{1,3}$ and  $w_{2,3}\notin P_{2,3}$, i.e., vertices $w_{1,2}, w_{1,3}, w_{2,3}$ are in subgraphs $H_1$, $H_2$, $H_3$. However, since no one of $H_1$, $H_2$, $H_3$ can have all three vertices $w_{1,2}, w_{1,3}, w_{2,3}$, for some $i,j\in \{1,2,3\}, i\neq j$, $d_G(V(H_i),V(H_j))\le 2r<K$ holds, which is impossible. 

Thus, $G$ cannot have any $K$-fat $K_3$-minor for $K>2r$.  
%Hence, by symmetry and distance requirements in (2), 
\qed 
\end{proof} 

Denote by $\mf(G)$ the largest $K>0$ such that $G$ has a $K$-fat $K_3$-minor. Call it the {\em $K_3$-minor fatness} of $G$.  We obtain the following theorem from Corollary    \ref{cor:mcw-adt}, Corollary  \ref{cor: adt-fm}, Lemma \ref{lm:mcw_fat}, Proposition \ref{prop:Papaaoglu}, Theorem \ref{th:bnc-delta-tl}, and Theorem  \ref{th:mcw-delta-rho}.

\begin{theorem} \label{th:mf-tl-mcw}
    For every graph $G$ and every vertex $s$ of $G$, $$\frac{\mf(G)}{2}\le \mc(G)\le\tb(G)\le\tl(G)\le \Delta_s(G)+1\le 5\cdot\mf(G),$$  $$\frac{\mf(G)-1}{2}\le \frac{2\cdot\mc(G)-1}{3}\le \adt(G)\le \Delta_s(G)\le 5\cdot \mf(G)-1,$$ $$\mf(G)\le 2\cdot\bn(G)+1\le \Delta_s(G)+1\le 5\cdot \mf(G).$$ 
\end{theorem}

%\begin{corollary}
%   For every graph $G$ and every vertex $s$ of $G$, $\Delta_s(G)\ge 5K$ then $G$ has a $K$-fat $K_3$-minor.  
%\end{corollary}

\subsection{Cycle bridging properties}\label{sec:cbc}
The following {\em characteristic cycle property} immediately follows from the definition of a chordal graph $G$: {\em For every simple cycle $C$ %$(|C|\ge 3)$ 
of a chordal graph $G$ and every vertex $v\in C$, the two neighbors in $C$ of $v$ are adjacent or there is a third vertex in $C$ that is adjacent to $v$.} We can generalize this property and show that its generalized version is coarsely equivalent to tree-length. 

Let $C$ be a simple cycle of $G$, %with $|C|\ge 2k+4$ ($k\ge 1$), 
$v$ be a vertex of $C$ and $x,y$ be two vertices of $C$ with $d_C(x,v)=d_C(y,v)=k$ (assuming $C$ is long enough). We call $x,y$ the {\em $k$-neighbors of $v$ in $C$}.  Denote by $\cbc(G)$ (call it the {\em cycle bridging constant} of $G$) the minimum $k$ ($k\ge 1$) such that for every simple cycle $C$ of $G$ %with $|C|\ge 2k+4$ 
and every vertex $v$ of $C$, if $d_G(x,y)=2k$ holds for the two $k$-neighbors $x,y$ of $v$ in $C$ (resulting in $|C|\ge 4k$), then there is a vertex $z\in C$ satisfying $d_G(v,z)\le k< d_C(v,z)$ (the latter inequality just says that $z$ is on $C\setminus P_v(x,y)$, where $P_v(x,y)$ is a part of $C$ between $x$ and $y$ containing $v$). Clearly, $\cbc(G)$ of a chordal graph $G$ is 1.  The following two lemmas hold. 

\begin{lemma}   \label{lm:bnc-cbc}
	For every graph $G$,  $\bn(G)\le \cbc(G)$. 
\end{lemma} 

\begin{proof}
Assume $\bn(G)>\cbc(G):=k$. By the definition of the bottleneck constant, for some vertices $x,y$ of $G$ at even distance from each other, some shortest path $P(x,y)$ connecting $x$ and $y$ and the middle vertex $s$ of $P(x,y)$,  there must exist a path $Q$ in $G[V\setminus D_{k}(s)]$ connecting $x$ and $y$. Necessarily, $d_G(x,s)=d_G(y,s)\ge k+1$. 
Consider a vertex $x'$  on a subpath of $P(x,y)$ between $x$ and $s$ and a vertex $y'$  on a subpath of $P(x,y)$ between $y$ and $s$ with $d_G(x',s)=d_G(y',s)=k+1$. We have that vertices $x'$ and $y'$ with $d_G(x',y')=2k+2$ and $d_G(x',s)=d_G(s,y')=k+1$ are connected in $G[V\setminus D_{k}(s)]$ by  $P(x,x')\cup Q\cup P(y,y')$, where $P(x,x')$ and $P(y,y')$ are subpaths of $P(x,y)$ between corresponding vertices. Extract from  $P(x,x')\cup Q\cup P(y,y')$ a simple subpath $Q'$ connecting $x'$ and $y'$ outside the disk $D_{k}(s)$. The union of $Q'$ and $P(x',y')$ forms a simple cycle in which the $k$-neighbors of $s$ are at distance $2k$ in $G$ but $s$ does not have any vertex in $Q'$ at distance at most $k$ in $G$. The latter contradicts with $\cbc(G)=k$. 
\qed
\end{proof}

\begin{lemma}   \label{lm:Delta-cbc}
	For every graph $G$, $\cbc(G)\le \lceil\frac{\widehat{\Delta}(G)+1}{2}\rceil\le \frac{\widehat{\Delta}(G)}{2}+1$. 
 % (OLD $\cbc(G)\le \frac{\widehat{\Delta}(G)+3}{2}$.) 
 %OLD-OLD $\cbc(G)\le \frac{\widehat{\Delta}(G)-1}{2}$. 
\end{lemma} 

\begin{proof}
Let $k:=\lceil\frac{\widehat{\Delta}(G)+1}{2}\rceil$, $C$ be a simple cycle of $G$,  $v$ be an arbitrary vertex of $C$, and assume that $d_G(x,y)=2k$ holds for the two $k$-neighbors $x,y$ of $v$ in $C$. Consider the layering partition $\mathcal{LP}(G,v)$ of $G$ starting at $v$. Since $d_G(x,v)=d_G(y,v)=k$, vertices $x$ and $y$ belong to the same layer $L^{k}$ of the layering of $G$ with respect to $v$. Since $d_G(x,y)=2k=2\lceil\frac{\widehat{\Delta}(G)+1}{2}\rceil>\widehat{\Delta}(G)$, by the definition of $\widehat{\Delta}(G)$, $x$ and $y$ cannot belong to the same cluster from $L^{k}$. By the definition of clusters, every path connecting $x$ with $y$ in $G$ must have a vertex in $D_{k-1}(v)$. Hence, there must exist also a vertex $z\in C\setminus P_v(x,y)$ such that $d_G(v,z)\le k-1<k$, where $P_v(x,y)$ is a part of $C$ between $x$ and $y$ containing $v$.  %Choosing now  $k=\frac{\widehat{\Delta}(G)-1}{2}$, we guarantee $d_G(x,y)=\widehat{\Delta}(G)+1>  \widehat{\Delta}(G)$. 
Consequently, $\cbc(G)\le \lceil\frac{\widehat{\Delta}(G)+1}{2}\rceil\le \frac{\widehat{\Delta}(G)}{2}+1$.  \qed
\end{proof}

Combining Lemma  \ref{lm:bnc-cbc} and Lemma  \ref{lm:Delta-cbc}  with Theorem \ref{th:bnc-delta-tl}, we get the following result. 

\begin{theorem} \label{th:cbc-tl-delta}
For every graph $G$, the following inequalities hold:  
 $$\frac{\tl(G)-3}{4}\leq\frac{\widehat{\Delta}(G)-2}{4}\leq \bn(G)\le \cbc(G)\leq\frac{\widehat{\Delta}(G)}{2}+1\leq\frac{3 }{2}\tl(G)+1.$$
\end{theorem} 

\begin{corollary} \label{cor:ineq-tl-cbc}
	For every graph $G$,   $\frac{2}{3}(\cbc(G)-1)\leq \tl(G)\leq 4\cdot \cbc(G)+3$. 	
\end{corollary}

%\begin{theorem} \label{th:cbc-tl-delta}
%For every graph $G$, the following inequalities hold:  
%$$\frac{\tl(G)-3}{4}\leq\frac{\widehat{\Delta}(G)-2}{4}\leq \bn(G)\le\cbc(G)\leq\frac{\widehat{\Delta}(G)+3}{2}\leq\frac{3 }{2}(\tl(G)+1).$$
%\end{theorem} 

%\begin{corollary} \label{cor:ineq-tl-cbc}
%For every graph $G$,   $\frac{2}{3}\cbc(G)-1\leq \tl(G)\leq 4\cdot \cbc(G)+3$. 	
%\end{corollary}

Note that, although $\tl(G)=\cbc(G)=1$ for every chordal graph $G$, generally, these two graph parameters are not equal. %$\tl(G)\neq \cbc(G)$. 
Consider a cycle $C_{12k}$ on $12k$ vertices. We have $\tl(C_{12k})=12k/3=4k$ and $\cbc(C_{12k})=12k/4+1=3k+1.$  Furthermore, since $\bn(C_4)=\bn(C_5)=\bn(C_6)=\bn(C_7)=1$ and $\cbc(C_4)=\cbc(C_5)=\cbc(C_6)=\cbc(C_7)=2$, generally, $\bn(G)$ and $\cbc(G)$  are not equal. %differ from each other.  
However, it turns out %we can show 
that, indeed, they are only one unit apart. 

\begin{lemma}   \label{lm:bnc-cbc-1}
	For every graph $G$,  $\bn(G)\le\cbc(G)\le \bn(G)+1$. 
\end{lemma} 

\begin{proof} By Lemma   \ref{lm:bnc-cbc}, we need only to show $\cbc(G)\le \bn(G)+1$.  Let $k:=\bn(G)+1$,  $C$ be a simple cycle of $G$,  $v$ be an arbitrary vertex of $C$, and assume that $d_G(x,y)=2k$ holds for the two $k$-neighbors $x,y$ of $v$ in $C$. By the definition of $\bn(G)$, vertex $v$ has at distance at most $\bn(G)=k-1$ a vertex in every path of $G$ connecting $x$ and $y$. Necessarily, there must exist a vertex $z$ in $C\setminus P_v(x,y)$ such that $d_G(v,z)\le k-1<k$, where $P_v(x,y)$ is a part of $C$ between $x$ and $y$ containing $v$.  %Choosing now  $k=\frac{\widehat{\Delta}(G)-1}{2}$, we guarantee $d_G(x,y)=\widehat{\Delta}(G)+1>  \widehat{\Delta}(G)$. 
Consequently, $\cbc(G)\le k=\bn(G)+1$.  \qed
\end{proof}
%{\color{red} 
% Lemma \ref{lm:Delta-cbc} can be obtained also as corollary of this lemma: $$\cbc(G)\le \bn(G)+1  \leq\frac{\widehat{\Delta}(G)}{2}+1$$
% }
 
We can now better relate $\cbc(G)$ to $\adt(G)$ and $\mc(G)$. By Corollary  \ref{cor:ineq-bnc-mcw}, Theorem  \ref{th:adt-tl-delta} and Lemma    \ref{lm:bnc-cbc-1}, we have the following inequalities. 

\begin{corollary} \label{cor:ineq-adt-mcw--cbc}
	For every graph $G$,   $\frac{\cbc(G)-2}{3}\leq \frac{\bn(G)-1}{3}\leq\adt(G)\leq 4\cdot \bn(G)+2\leq 4\cdot \cbc(G)+2$ and $\frac{\cbc(G)-1}{3}\leq \mc(G)\leq 4\cdot \cbc(G)+2$. 	
\end{corollary}

\commentout{
{\color{red}/* see in Diestel et al. Conjecture 1.6. and Th.1.1. (which is proved by Lm.3.6 and Lm.3.7).    \\
- my Lm \ref{lm:bnc-cbc} is like their conjecture: \\
--- if $\tl$ is large then there is a bad cycle [if no bad cycle then $\tl$ is bounded] \\
- their Th.1.1. (is like my Lm \ref{lm:Delta-cbc}?): \\
--- if there is a bad cycle then $\tl$ is unbounded [if $\tl$ is bounded then no bad cycle] */ \\

if $\tl\le 4/3\cdot \cbc$ is true like for cycles, this gives a chance to get \~ 2-appr for computing $\tl$ ??? }\\
}

We can define one more condition on cycles which also  turns out to be coarsely equivalent to tree-length. Recall that a simple cycle $C$ of $G$ is called {\em geodesic} if for every $x,y\in C$, $d_G(x,y)=d_C(x,y)$. Let us call a simple cycle $C$ of $G$ {\em $\mu$-locally geodesic} if for every two vertices $x,y\in C$, $d_G(x,y)=d_C(x,y)=\ell$ implies $\ell\le \mu$. When $\mu\ge |C|/2$, clearly, every   $\mu$-locally geodesic cycle $C$ is geodesic. If a simple cycle $C$ is not $\mu$-locally geodesic in $G$, then there must exist two vertices in $C$ such that $d_G(x,y)=d_C(x,y)=\mu+1$. We call that side of $C$ between $x$ and $y$ which realizes $d_C(x,y)$ a  {\em side of non-$\mu$-locality} (note that $x$ and $y$ both belong to that side). We say that vertices $v,z$ of $C$ form a {\em $k$-bridge} in $C$ if $d_G(v,z)\le k<d_C(v,z)$. 

Denote by $\bgc(G)$ (call it the {\em ''bridging non-locally geodesic cycles`` constant} of $G$) the minimum $\mu$ ($\mu\ge 1$) such that for every cycle $C$ of $G$ %with $|C|\ge 2k+4$ 
that is not $\mu$-locally geodesic, from every side of non-$\mu$-locality there is a %????$\lceil\mu/2\rceil$-bridge 
$\lfloor \frac{\mu+1}{2}\rfloor$-bridge to other side of $C$. 
%{\color{blue} Notice that this notion generalizes another {\em characteristic  cycle property} of a chordal graph $G$: {\em For every simple cycle $C$ %$(|C|\ge 3)$ 
%of a chordal graph $G$ and every edge $vu\in C$, there is a third vertex in $C$ that forms a triangle with $vu$.}}  

\begin{lemma}   \label{lm:bgc-cbc}
	For every graph $G$,  $\bgc(G)\le 2\cdot\cbc(G)-1$. 
\end{lemma} 

\begin{proof}  Let $\cbc(G)=k$ and $C$ be a cycle of $G$ that is not $(2k-1)$-locally geodesic. Let $x,y\in C$  such that $d_G(x,y)=d_C(x,y)=2k$. Consider vertex $v$ of $C$ between $x$ and $y$ with  $d_C(v,x)=k=d_C(v,y)$. The two $k$-neighbors of $v$ in $C$ are at distance $2k$ from each other in $G$. By the definition of $\cbc(G)$, there is a vertex $z\in C$ such that $d_G(v,z)\le k<d_C(v,z)$. Consequently, $\bgc(G)\le 2\cdot\cbc(G)-1$. 
\qed
\end{proof}

Now, using a layering partition of a graph $G$, we upperbound $\Delta_s(G)$ by a linear function of $\bgc(G)$.

\begin{lemma} \label{lm:ClusterDiam_bgc}
	For every graph $G$  and every vertex $s$ of $G$, $\Delta_s(G)\leq 4\cdot \bgc(G)+6$. In particular, $\widehat{\Delta}(G)\leq 4\cdot
 \bgc(G)+6$ for every graph $G$.
\end{lemma}

\begin{proof}
Let $s$ be an arbitrary vertex of $G$ and $\mathcal{LP}(G,s)$ be the layering partition of $G$ starting at $s$. Consider vertices $x$ and $y$ from a cluster of $\mathcal{LP}(G,s)$ with $d_G(x,y)=\Delta_s(G)$, and let $k:=d_G(s,x)=d_G(s,y)$ and $\mu:=\bgc(G)+1$. Choose also a path $Q$ connecting $x$ and $y$ outside the disk $D_{k-1}(s)$.

Consider arbitrary  shortest paths $P(s,x)$ and $P(s,y)$  of $G$ connecting $s$ with $x$ and $y$, respectively. Let $s'$ be a vertex from  $P(s,x)\cap P(s,y)$ furthest from $s$. We may assume that $k':=d_G(x,s')=d_G(y,s')=k-d_G(s,s')$ is greater than $\mu+\lfloor \frac{\mu}{2}\rfloor$ %\lceil\mu/2\rceil$ 
since, otherwise,  $d_G(x,y)\le d_G(x,s')+d_G(s',y)=2k'\leq 3\mu$, and we are done. 

Pick now vertices $u,w$ in $P(s,x)$ such that  
$d_G(x,u)=\mu+\lfloor \frac{\mu}{2}\rfloor%\lceil\mu/2\rceil
+1\leq d_G(x,s')$, $d_G(u,w)=\mu$ and $d_G(x,w)=\lfloor \frac{\mu}{2}\rfloor%\lceil\mu/2\rceil
+1$. Let $P(u,w)$ be a subpath of $P(s,x)$ between $u$ and $w$. 
%and a shortest path $P(s',y)$ connecting $s'$ with $y$.  
%
A simple cycle $C$ of $G$ formed by $Q$ and subpaths $P(s',x)$ and $P(s',y)$ of $P(s,x)$ and $P(s,y)$, respectively, is not $(\mu-1)$-locally geodesic. Since $\mu-1=\bgc(G)$, for a side $P(u,w)$ of $C$ of non-$(\mu-1)$-locality, we must have a $\lfloor\mu/2\rfloor$-bridge from a vertex $v\in P(u,w)$ to a vertex $z\in C\setminus P(u,w)$. As $d_G(v,z)\le \lfloor\mu/2\rfloor<d_C(v,z)$, vertex $z$ cannot belong to a shortest path $P(s',x)$ (which is a part of $C$ and contains also $v$). If $z$ belongs to  $Q$, then $d_G(s,Q)=k=d_G(s,x)=d_G(s,v)+d_G(v,z)\le d_G(s,w)+\lfloor\mu/2\rfloor<d_G(s,w)+d_G(w,x)=d_G(s,x)=k$, giving a contradiction. Thus, $z$ cannot be in $Q$ either. Consequently, $z\in P(s',y)$.

Since $d_G(y,z)=k-d_G(s,z)$ and $k=d_G(x,s)\le d_G(s,z)+d_G(z,v)+d_G(v,x)$, necessarily, $d_G(y,z)\le d_G(z,v)+d_G(v,x)\le d(z,v)+d_G(u,x)\le \lfloor\mu/2\rfloor+\mu+\lfloor\mu/2\rfloor+1\le 2\mu+1$. Consequently,  $d_G(x,y)\le d_G(x,v)+d_G(v,z)+d_G(z,y)\leq d_G(x,u)+d_G(v,z)+d_G(z,y)\le \mu+\lfloor\mu/2\rfloor+1+\lfloor\mu/2\rfloor+ 2\mu+1=4\mu+2$. That is,  $\Delta_s(G)\leq 4 (\bgc(G)+1)+2= 4\cdot \bgc(G)+6$. \qed
\end{proof}

Summarizing (using Proposition \ref{prop:dorisb}, Lemma  \ref{lm:Delta-cbc}, Lemma   \ref{lm:bgc-cbc}, and Lemma \ref{lm:ClusterDiam_bgc}), we conclude. 
\begin{theorem} \label{th:bgc-tl-delta}
For every graph $G$, the following inequalities hold:  
 $$\frac{\tl(G)-7}{4}\leq\frac{\widehat{\Delta}(G)-6}{4}\leq \bgc(G)\le 2\cbc(G)-1\leq\widehat{\Delta}(G)+1\leq 3\cdot \tl(G)+1.$$
\end{theorem} 

\begin{corollary} \label{cor:ineq-tl-bgc}
	For every graph $G$,   $\frac{\bgc(G)-1}{3}\leq \tl(G)\leq 4\cdot \bgc(G)+7$. 	
\end{corollary}

From Lemma \ref{lm:Delta-cbc} and Lemma  \ref{lm:ClusterDiam_bgc}, it follows $\cbc(G)\le 2\cdot \bgc(G)+4$. We can show a better bound directly. 

\begin{lemma} \label{lm:cbc_bgc}
	For every graph $G$, $\cbc(G)\leq \bgc(G)+1$. 
\end{lemma}
\begin{proof} Let $k:=\lfloor\frac{\bgc(G)}{2}\rfloor+1+\lfloor\frac{\bgc(G)+1}{2}\rfloor$,  $C$ be a simple cycle of $G$,  $v$ be an arbitrary vertex of $C$, and assume that $d_G(x,y)=2k$ holds for the two $k$-neighbors $x,y$ of $v$ in $C$. We have $d_G(x,y)=d_C(x,y)=2k$. Consider the shortest path $P(x,y)\subset C$ between $x$ and $y$, and let $x',y'$ be two distinct vertices of $P(x,y)$ at distance $\lfloor\frac{\bgc(G)}{2}\rfloor+1$ from $v$. Let $P(x',y')$ be a subpath of $P(x,y)$ between $x'$ and $y'$. Its length $d_G(x',y')$ is $2(\lfloor\frac{\bgc(G)}{2}\rfloor+1)\ge\bgc(G)+1$. 
By the definition of $\bgc(G)$, from non-$\bgc(G)$-locality  side $P(x',y')$ of $C$ there must exist a $\lfloor\frac{\bgc(G)+1}{2}\rfloor$-bridge to other side of $C$, i.e., vertices $v'\in P(x',y')$ and $z\in C\setminus P(x',y')$ such that $d_G(v',z)\le \lfloor\frac{\bgc(G)+1}{2}\rfloor<d_C(v',z)$. The inequality  $d_G(v',z)\le \lfloor\frac{\bgc(G)+1}{2}\rfloor$ guarantees $d_G(v,z)\le d_G(v,v')+d_G(v',z)\le \lfloor\frac{\bgc(G)}{2}\rfloor+1+\lfloor\frac{\bgc(G)+1}{2}\rfloor=k$. The inequality  $d_G(v',z)<d_C(v',z)$  guarantees that $z$ cannot be in $P(x,y)$ (recall that $P(x,y)\subset C$ is a shortest path of $G$ and, hence, for every $u,w\in P(x,y)$, $d_G(u,w)=d_C(u,w)$). 

Consequently, $\cbc(G)\le k=\bgc(G)+1$.  \qed
\end{proof}

%{\color{red} 
% Lemma \ref{lm:ClusterDiam_bgc} can be obtained also as corollary of this lemma: $$\Delta_s(G)\leq 4\cdot \bn(G)+2\leq 4\cdot \cbc(G)+2\leq 4\cdot \bgc(G)+6.$$}
 
Combining Lemma  \ref{lm:bnc-cbc-1}, Lemma  \ref{lm:bgc-cbc}, Lemma  \ref{lm:cbc_bgc}  and Corollary  \ref{cor:ineq-adt-mcw--cbc}, we get the following inequalities. %corollary. 

\begin{corollary} \label{cor:ineq-adt-mcw-cbc--bgc}
	For every graph $G$,  $$\bn(G)\le \cbc(G)\leq \bgc(G)+1\le 2\cdot\cbc(G)\le 2\cdot\bn(G)+2,$$ 
 $$\frac{\bgc(G)-3}{6}\leq \frac{\bn(G)-1}{3}\leq\adt(G)\leq 4\cdot \bn(G)+2\leq 4\cdot \bgc(G)+6,$$ $$\frac{\bgc(G)-1}{6}\leq \mc(G)\leq 4\cdot \bgc(G)+6.$$ 	
\end{corollary}

%-------------------------------------

\section{Concluding remarks and open questions} \label{sec:concl}
%{\color{red}
%- improved bounds \\
%- more coarse parameters \\}
We saw that several graph parameters are coarsely equivalent to tree-length. If one of the parameters from the list $\{\tl(G),\tb(G), \itl(G), \itb(G), \Delta_s(G), \rho_s(G),  \td(G), \ad(G), \adt(G), $ $\bn(G), \mc(G), $ $\ph(G),\sh(G),\br(G), \mf(G), \glc(G), \cbc(G), \bgc(G)\}$ is bounded for a graph $G$, then all other parameters are bounded. We saw that, in fact, all those parameters are within small constant factors from each other. %Two immediate questions are in order.  
Two questions are immediate.  
\begin{itemize}
\item[1.]   Can constants in those inequalities be further improved? Is $\br(G)=\tb(G)$  true? 
\item[2.] Are there any other interesting graph parameters that are coarsely equivalent to tree-length? 
\end{itemize}

%- cycle conjecture \\

We can further relax the  ''bridging non-locally geodesic cycles`` condition on cycles. Let $\BGC(G)$ be the minimum $\mu$ ($\mu\ge 1$) such that every simple cycle $C$ of $G$ %with $|C|\ge 2k+4$ 
that is not $\mu$-locally geodesic has a 
$\lfloor \frac{\mu+1}{2}\rfloor$-bridge (in general position). Clearly,  $\BGC(G)\leq \bgc(G)$. 
\begin{itemize}
\item[3.]   Does there exist a constant $c$ such that $\bgc(G)\le c\cdot\BGC(G)$ for every graph $G$? 
\end{itemize}

%- cite Arne's dissertation \\
%- Strong tree-breadth. \\

In \cite{st-tb}, a notion of strong tree-breadth was introduced. The {\em  strong breadth}  of a tree-decomposition $\cT(G)$ of a graph $G$ is the minimum integer $r$ such that for every $B\in V(\cT(G))$ there is a vertex $v_B\in B$ with $B= D_r(v_B,G)$ (i.e., each bag $B$ is equal to a disk  of $G$ of radius at most $r$). The {\em strong tree-breadth}  of $G$, denoted by $\stb(G)$, is the minimum of the strong breadth, over all tree-decompositions of $G$.  Like for the tree-breadth, it is NP-complete to determine if a given graph has strong tree-breadth $r$, even for $r=1$ \cite{st-tb}. See \cite{ArneDisser,st-tb} for other interesting results on tree-breadth and strong tree-breadth. Clearly, $\tb(G)\le \stb(G)$ for every graph $G$. The following question was already asked in \cite{st-tb} and matches very much the topic of this paper.  
\begin{itemize}
\item[4.]   Does there exist a constant $c$ such that $\stb(G)\le c\cdot\tb(G)$ for every graph $G$? 
\end{itemize}

%-  $k$-tree-breadth, $k\ge 2$, is defined in . Can similar results be proven for  $k$-tree-breadth?  \\

In \cite{CTS1}, a notion generalizing tree-width and tree-breadth was introduced. The {\em $k$-breadth} of a tree-decomposition $\cT(G)$  of a graph $G$ is the minimum integer $r$ such
that for each bag $B\in V(\cT(G))$, there is a set of at most $k$ vertices $C_B = \{v_B^1,\dots,v_B^k\}\subseteq V(G)$ such that for each $u\in B$, $d_G (u, C_B) \le r$ holds (i.e., each bag $B$ can be covered with at most $k$ disks of $G$ of radius at most $r$ each; $B\subseteq \cup_{i=1}^k D_r(v_B^i,G)$. The {\em $k$-tree-breadth} of a graph $G$, denoted by $\tb_k(G)$, is the minimum of the $k$-breadth,
over all tree-decompositions of $G$. Clearly, for every graph $G$, $\tb(G) = \tb_1(G)$ and $\tw(G) \le k - 1$ if and only if $\tb_k(G) = 0$ (each vertex in the bags of the tree-decomposition can be considered as a center of a disk of radius 0). 
\begin{itemize}
\item[5.]   It would be interesting to investigate if there exist  generalizations of  some graph parameters considered in this paper that coarsely describe the $k$-tree-breadth.  
\end{itemize}

In a follow-up paper \cite{coarse-pathlength}, we introduce several graph parameters that are coarsely equivalent to path-length.

% - See https://link.springer.com/article/10.1007/s00454-011-9386-0 p. 196 on quasi-isometricity to trees 

\section*{Data Availability Statement}

Data sharing is not applicable to this article as no datasets were generated or analyzed during the current study.

\newpage %\centerline
{\bf Appendix A: Graph parameters considered} 
%\medskip

\begin{table} [htbp]
	\centering
	\begin{tabular}{| l | l |}
		\hline
Notation      & Name \\ 
\hline
\noalign{\smallskip}
 $\tl(G)$, $\itl(G)$ & tree-length, inner tree-length of $G$ \\ 
\hline
\noalign{\smallskip}
 $\tb(G)$, $\itb(G)$ & tree-breadth, inner tree-breadth of $G$ \\ 
\hline
\noalign{\smallskip}
 $\stb(G)$ & strong tree-breadth of $G$ \\  
\hline
\noalign{\smallskip}
 $\br(G)$  & bramble interception radius  of $G$  \\  
\hline
\noalign{\smallskip}
$\sh(G)$, $\ph(G)$  & interception radius for Helly families of connected subgraphs or of  paths of $G$ \\  
\hline
\noalign{\smallskip}	
$\Delta_s(G)$ & cluster-diameter of a layering partition of $G$ with respect to a start vertex $s$ \\
\hline
\noalign{\smallskip}	 
$\Delta(G)$, $\widehat{\Delta}(G)$~~ & minimum and maximum cluster-diameter over all layering partitions of $G$ \\
\hline 
\noalign{\smallskip}	
$\rho_s(G)$ & cluster-radius of a layering partition of $G$ with respect to a start vertex $s$ \\
\hline  
\noalign{\smallskip}	 
$\rho(G)$, $\widehat{\rho}(G)$ & minimum and maximum cluster-radius over all layering partitions of $G$ \\
\hline 
\noalign{\smallskip}	
$\td(G)$ & non-contractive multiplicative distortion of embedding of $G$ into a weighted tree \\
\hline 
\noalign{\smallskip}	
$\ad(G)$ & additive distortion of embedding of $G$ into a weighted tree\\
\hline 
\noalign{\smallskip}	
$\adt(G)$ & additive distortion of embedding of $G$ to an unweighted tree \\
\hline 
\noalign{\smallskip}	
$\bn(G)$, $\BNC(G)~$ & bottleneck constant, overall bottleneck constant of $G$ \\
\hline 
\noalign{\smallskip}	
$\mc(G)$ & McCarty-width of $G$ \\
\hline 
\noalign{\smallskip}	
$\mc_k(G)$ & McCarty-width of order $k$ of $G$ \\
\hline 
\noalign{\smallskip}	
$\mf(G)$ & $K_3$-minor fatness of $G$ \\
\hline 
\noalign{\smallskip}	
$\glc(G)$ &  maximum load over all geodesic loaded cycles in $G$ \\
\hline 
\noalign{\smallskip}	
$\cbc(G)$ & cycle bridging constant of $G$ \\
\hline 
\noalign{\smallskip}	
$\bgc(G)$ & ''bridging non-locally geodesic cycles`` constant of $G$ \\
\hline 
\end{tabular}
%\vspace*{7mm}
%	\caption{Graph parameters.}
	\label{table:parameters}
\end{table}
%}$$ be the maximum load over all geodesic loaded cycles in $G$.

{\bf Appendix B: Bounds known before} 
%\medskip

$$\tl(G) \leq \itl(G)\le 2\cdot\tl(G)  \mbox{~\cite{BerSey2024}~~and~~} \tb(G) \leq \itb(G)\le 2\cdot\tb(G)\mbox{~\cite{Diestel++}}$$ 
$$\rho_s(G) \leq \Delta_s(G) \leq 2\rho_s(G) \mbox{~~ and~~ } \tb(G) \leq \tl(G) \leq 2\tb(G)  \mbox{~~[trivial]}$$ 
$$\Delta(G)\le  \Delta_s(G) \leq\widehat{\Delta}(G)\le 3 \Delta(G) \mbox{~~  \cite{slimness} (see Proposition \ref{prop:ClustDiamAtAnys})}$$ 
	$$\frac{\Delta_s(G)}{3} \leq \tl(G) \leq \Delta_s(G)+1 \mbox{~~  \cite{Dorisb2007} (see Proposition \ref{prop:dorisb})}$$ 
	$$\frac{\rho_s(G)}{3} \leq \tb(G) \leq \rho_s(G)+1 \mbox{~~  \cite{AbDr16,tree-spanner-appr} (see Proposition \ref{prop:Muad-Feodor})}$$ 
	$$\frac{\Delta_s(G)}{3} \leq \tl(G) \leq \td(G) \le 2\cdot \Delta_s(G)+2 \mbox{~~  \cite{AbDr16,ChepoiDNRV12,tree-spanner-appr} (see Proposition \ref{prop:td-tl})}$$ 
$$\frac{\tl(G)-2}{2}\leq \ad(G)\leq 6\cdot \tl(G)\mbox{~~ \cite{BerSey2024} (see Proposition \ref{prop:ad-tl-Seymour})}$$
$$\frac{2}{3}\bn(G)\leq \tl(G)\leq 4\cdot \bn(G)+3 \mbox{~ and ~} 
\ad(G)\le 24\cdot \bn(G)+18 \mbox{~~  \cite{BerSey2024} (see Proposition \ref{pr:Seymour--Manning})} $$ 	
  $$\frac{\tl(G)- 3}{6} \le \mc(G) \le \tl(G) \mbox{~~ \cite{BerSey2024} (see Proposition \ref{prop:sey-mcw})}$$
$$\tl(G)- 1 \le \glc(G) \le 3\cdot\tl(G) \mbox{~~~ \cite{BerSey2024} (see Proposition \ref{prop:sey-lgc})}$$ 
$$\mf(G)\le 2\cdot\bn(G)+1 \mbox{~~and~~} \ad(G)\le 14\cdot\mf(G) \mbox{~~~ \cite{GeorPapa2023}}$$ 

\newpage %\centerline
{\bf Appendix C: Bounds %(re-) proved in 
from this paper} 
\bigskip

{\bf Bounds with $\bn(G)$} (Theorem \ref{th:bnc-delta-tl}, Corollary \ref{cor:ineq-tl-bnc}). 
  $$\frac{\Delta_s(G)-2}{4}\leq \bn(G)= \BNC(G)\leq \frac{3}{2}\Delta_s(G)$$ 
   $$\frac{\tl(G)-3}{4}\leq\frac{\widehat{\Delta}(G)-2}{4}\leq \bn(G)=\BNC(G)\leq\frac{\widehat{\Delta}(G)}{2}\leq\frac{3 }{2}\tl(G)$$
 $$\frac{2}{3}\bn(G)\leq \tl(G)\leq 4\cdot \bn(G)+3$$
%\bigskip\center{***}

{\bf Bounds with $\mc(G)$ (with $\ph(G),\sh(G),\br(G)$)} (Theorem \ref{th:mcw-delta-rho}, Corollary \ref{cor:ineq-bnc-mcw}, Proposition \ref{prop:bramble}).  
 $$\mc(G)\leq \rho(G)\le\rho_s(G)\le \Delta_s(G)\le\widehat{\Delta}(G)\leq 6\cdot \mc(G)$$ 
$$\mc(G)\le \ph(G)\le\sh(G)\le\br(G)\le \tb(G)\leq \tl(G)\leq  \Delta_s(G)+1\leq 6\cdot \mc(G)+1$$ 	
 $$\mc(G)\leq  4\cdot \bn(G)+2 \mbox{~~and~~} \bn(G)\le  3\cdot \mc(G)$$
%\bigskip\center{***}

{\bf Bounds with $\adt(G)$} (Theorem \ref{th:adt-tl-delta}, Corollary \ref{cor:ineq-tl-adt}, Corollary \ref{cor:mcw-adt}).  
 $$\adt(G)\leq \Delta(G)\le \Delta_s(G)\le\widehat{\Delta}(G)\leq 3\cdot \tl(G)\le 6\cdot \adt(G)+3$$ 
$$\adt(G)\leq \Delta_s(G)\le 4\cdot \bn(G)+2 \mbox{~~and~~} \bn(G)\le 3\cdot \adt(G)+1$$

$$\frac{\tl(G)-1}{2}\leq \adt(G)\leq 3\cdot \tl(G) \mbox{~~and~~}  \frac{\adt(G)}{6}\le \mc(G)\le \frac{3\cdot \adt(G)+1}{2}$$
%\bigskip\center{***}

{\bf Bounds with $\mf(G)$} (Theorem \ref{th:mf-tl-mcw}).
$$\frac{\mf(G)}{2}\le \mc(G)\le\tb(G)\le\tl(G)\le \Delta_s(G)+1\le 5\cdot\mf(G)$$  $$\frac{\mf(G)-1}{2}\le \frac{2\cdot\mc(G)-1}{3}\le \adt(G)\le \Delta_s(G)\le 5\cdot \mf(G)-1$$ $$\mf(G)\le 2\cdot\bn(G)+1\le \Delta_s(G)+1\le 5\cdot \mf(G)$$ 

{\bf Bounds with $\cbc(G)$} (Theorem \ref{th:cbc-tl-delta}, Corollary \ref{cor:ineq-tl-cbc}, Corollary \ref{cor:ineq-adt-mcw--cbc}).  
 $$\frac{\tl(G)-3}{4}\leq\frac{\widehat{\Delta}(G)-2}{4}\leq \bn(G)\le \cbc(G)\leq\frac{\widehat{\Delta}(G)}{2}+1\leq\frac{3 }{2}\tl(G)+1$$
%  $$\bn(G)\le\cbc(G)\le \bn(G)+1$$ 
  $$\frac{\cbc(G)-2}{3}\leq \frac{\bn(G)-1}{3}\leq\adt(G)\leq 4\cdot \bn(G)+2\leq 4\cdot \cbc(G)+2$$ $$\frac{\cbc(G)-1}{3}\leq \mc(G)\leq 4\cdot \cbc(G)+2$$
%\bigskip\center{***}

{\bf Bounds with $\bgc(G)$} (Theorem \ref{th:bgc-tl-delta},  Corollary \ref{cor:ineq-tl-cbc}, Corollary \ref{cor:ineq-tl-bgc}, Corollary \ref{cor:ineq-adt-mcw-cbc--bgc}).  
 $$\frac{\tl(G)-7}{4}\leq\frac{\widehat{\Delta}(G)-6}{4}\leq \bgc(G)\le 2\cbc(G)-1\leq\widehat{\Delta}(G)+1\leq 3\cdot \tl(G)+1$$
$$\frac{\bgc(G)-1}{3}\leq   \frac{2}{3}(\cbc(G)-1)\leq \tl(G)\leq 4\cdot \cbc(G)+3\leq 4\cdot \bgc(G)+7$$
  $$\bn(G)\le \cbc(G)\leq \bgc(G)+1\le 2\cdot\cbc(G)\le 2\cdot\bn(G)+2$$ 
 $$\frac{\bgc(G)-3}{6}\leq \frac{\bn(G)-1}{3}\leq\adt(G)\leq 4\cdot \bn(G)+2\leq 4\cdot \bgc(G)+6$$ $$\frac{\bgc(G)-1}{6}\leq \mc(G)\leq 4\cdot \bgc(G)+6$$ 	
%%%%%%%%%%%%%%%%%%%
\end{document}